\documentclass[11pt]{amsart}

\usepackage{verbatim, graphicx,amssymb,hyperref}

\usepackage{accents}
\usepackage{color}
\usepackage{soul}

\newcommand{\black}{\color[rgb]{0,0,0}}

\newcommand{\eps}{\varepsilon}

\newcommand{\scp}[2]{\langle #1, #2 \rangle}

\newcommand{\cF}{\mathcal F}

\newcommand{\cU}{\mathcal U}

\newcommand{\cost}{\operatorname{cost}}

\newcommand{\sfrac}[2]{\mbox{$\frac{#1}{#2}$}}

\newcommand{\sig}{\sigma}%{\widetilde \sigma}
\newcommand{\vel}{v}%{\widetilde v}
\renewcommand{\P}{P}
\renewcommand{\Pr}{{\mathbb P}}

\newcommand{\N}{{\mathbb N}}

\newcommand{\E}{{\mathbb E}}

\newcommand{\R}{{\mathbb R}}

\newcommand{\proj}{\operatorname{proj}}

\newcommand{\pr}{\operatorname{pr}}

\newcommand{\assCa}{\hyperref[C1]{C.1}}
\newcommand{\assCb}{\hyperref[C2]{C.2}}
\newcommand{\assAa}{\hyperref[A1]{A.1}}
\newcommand{\assBa}{\hyperref[B1]{B.1}}
\theoremstyle{plain}
\newtheorem{theorem}{Theorem}[section]
\newtheorem{prop}[theorem]{Proposition}
\newtheorem{lemma}[theorem]{Lemma}
\newtheorem{cor}[theorem]{Corollary}

\newtheorem{ext}[theorem]{Extension}
\theoremstyle{definition}
\newtheorem{rem}[theorem]{Remark}

\begin{document}

\title[Multilevel stochastic approximation]%
{General multilevel adaptations for stochastic approximation algorithms}

\author[]%[Dereich]
{Steffen Dereich}
\address{Steffen Dereich\\
Institut f\"ur Mathematische Statistik\\
Fachbereich 10: Mathematik und Informatik\\
Westf\"alische Wilhelms-Universit\"at M\"unster\\
Orl\'eans-Ring 10\\
48149 M\"unster\\
Germany}
\email{steffen.dereich@wwu.de}

\author[]%[M\"uller-Gronbach]
{Thomas M\"uller-Gronbach}
\address{Thomas M\"uller-Gronbach\\
Fakult\"at f\"ur Informatik und Mathematik\\
Universit\"at Passau\\
Innstra{\ss}e 33\\
94032 Passau\\
Germany}
\email{thomas.mueller-gronbach@uni-passau.de}

\keywords{Stochastic approximation; Monte Carlo; multilevel}
\subjclass{Primary 62L20; Secondary 60J10, 65C05}
%\date{}

\begin{abstract}
In this article we present and analyse new multilevel adaptations of stochastic approximation algorithms for the computation of a zero of a function $f\colon D \to \R^d$ defined on a convex domain  $D\subset \R^d$, which is given as a parameterised family of expectations. Our approach is universal in the sense that having multilevel implementations for a particular application at hand it is straightforward to implement the corresponding stochastic approximation algorithm. Moreover, previous research on multilevel Monte Carlo can be incorporated in a natural way. 
This is due to the fact that the analysis of the error and the computational cost of our method is  based on similar assumptions as used in Giles~\cite{Gil08} for the computation of a single expectation. Additionally, we essentially only require that $f$ satisfies a classical contraction property from stochastic approximation theory. Under these assumptions we establish error bounds in $p$-th mean for our multilevel Robbins-Monro and Polyak-Ruppert schemes that decay 
in the computational time as fast as the classical error bounds for multilevel Monte Carlo approximations of single expectations known from Giles~\cite{Gil08}.
\end{abstract}

\maketitle

\section{Introduction}\label{s1}

Let  $D\subset \R^d$  be closed and  convex  and  let $U$ be  a random variable on a  probability space $(\Omega,\cF,\Pr)$
with values in a set $\mathcal U$ equipped with some $\sigma$-field.
We study the problem of computing  zeros of functions $f\colon D\to\R^d$ 
of the form
\[
f(\theta)=\E[F(\theta,U)], %=\int F^\omega(\theta)\,\dd \P(\omega),
\]
where $F\colon D\times \mathcal U\to \R^d$ is a product measurable function such that all expectations  $\E[F(\theta,U)]$  are well-defined.  
In this article we focus on the case where the random variables  $F(\theta,U)$ cannot be  simulated directly so that one has to work with appropriate approximations  in numerical simulations.  For example, one may think of $U$ being a Brownian motion and of $F(\theta,U)$ being the payoff of an option, where $\theta$ is a parameter affecting the payoff and/or the dynamics of the price process. Alternatively, $F(\theta,U)$ might be the value of a PDE at certain positions with $U$ representing random coefficients and $\theta$ a parameter of the equation.

In previous years the multilevel paradigm introduced by Heinrich~\cite{Hei01} and Giles~\cite{Gil08}  has proved to be a very efficient tool in the numerical computation of expectations. By Frikha~\cite{Fri15} it has recently been shown 
%for the first time 
that the efficiency of the multilevel paradigm  prevails  when combined with  stochastic approximation algorithms. In the present paper we take a different
% and more straightforward 
approach than the one introduced by the latter author. Instead of employing a sequence of coupled Robbins-Monro algorithms to construct a multilevel estimate of a zero of $f$ we basically propose a single Robbins-Monro algorithm that uses in the $(n+1)$-th  step a multilevel estimate of $\E[F(\theta_n,U)]$ with a complexity that is adapted to the actual state $\theta_n$ of the system and increases in the number of steps. 

Our approach is universal in the sense that having multilevel implementations for a particular application at hand it is straightforward to implement the corresponding stochastic approximation algorithm. Moreover, previous research on multilevel Monte Carlo can be incorporated in a natural way. This is due to the fact that
the analysis of the error and the computational cost of our method is  based on similar assumptions on the biases, the $p$-th central moments and the simulation cost of the underlying approximations   of $F(\theta,U)$  as used in Giles~\cite{Gil08}, see Assumptions \assCa\ and \assCb\  in Section~\ref{sec3}.  Additionally,  we require that $f$ satisfies a classical contraction property from stochastic approximation theory:  there exist $L>0$ and a zero $\theta^*$ of $f$ such that for all $\theta\in D$, 
\[
\langle f(\theta),\theta-\theta^*\rangle\leq -L \|\theta-\theta^*\|^2,
\]
where $\langle \cdot,\cdot\rangle$ denotes an inner product on $\R^d$. Moreover, $f$ has to satisfy a linear growth condition relative to the zero $\theta^*$, see Assumption \assAa\ and Remark~\ref{r001} in Section~\ref{sec2}. Note that the contraction property  implies that the zero $\theta^*$ is unique. Theorem~\ref{thm_Gilesnew} asserts that under these assumptions the maximum $p$-th mean error $\sup_{k\ge n}\E[\|\theta_k-\theta^*\|^p]$ of our properly tuned multilevel Robbins-Monro scheme $(\theta_n)_{n\in\N}$  satisfies the same upper bounds in terms of the computational time needed to compute $\theta_n$ as the bounds obtained in Giles~\cite{Gil08} for the multilevel computation of a single expectation.

In general, the design of this algorithm requires knowledge on the constant $L$ in the contraction property of $f$. To bypass this problem without loss of efficiency one may 
work with a Polyak-Ruppert average of our algorithm. Theorem~\ref{thm_Gilesnew_2} states that under Assumptions \assCa\ and 
\assCb\ on the  approximations  of $F(\theta,U)$ and Assumption \assBa\ on $f$, which is slightly stronger than condition \assAa\,  a properly tuned multilevel Polyak-Ruppert average $(\bar \theta_n)_{n\in\N}$ achieves, for $q<p$, the same upper bounds in the relation of the $q$-th mean error $\E[\|\bar \theta_n-\theta^*\|^q]$  and the corresponding computational time as the previously introduced multilevel Robbins-Monro method.

We briefly outline the content of the paper. The multilevel algorithms and the respective complexity theorems are presented in Section~\ref{sec3} for the case where $D=\R^d$. General  closed convex  domains $D$ are covered in Section~\ref{sec4}. We add that Sections~\ref{sec3} and~\ref{sec4} are self-contained and a reader interested  in the multilevel schemes only, can immediately start reading in Section~\ref{sec3}.

The error analysis of the multilevel stochastic approximation algorithms is based on new estimates of the $p$-th mean error of Robbins-Monro and Polyak-Ruppert algorithms. These 
results are presented in Section~\ref{sec2}. As a technical tool we employ a 
modified Burkholder-Davis-Gundy inequality, which is established in the appendix and 
might be of interest in itself, see Theorem~\ref{thm_BDG}.

We add that formally all results of the following sections remain true when replacing $(\R^d,\langle\cdot,\cdot\rangle)$ by an arbitrary separable Hilbert space. However in that case the definition~\eqref{cost1} of the computational cost of a multilevel algorithm might not be appropriate in general.

 \section{New error estimates for stochastic approximation algorithms}\label{sec2}

Since the pioneering work of  Robbins and Monro~\cite{RM51} in 1951 a large body of research has been devoted  to the analysis of stochastic approximation algorithms  with a strong focus on pathwise and weak convergence properties. In particular, laws of iterated logarithm and central limit theorems have been established that allow to optimise the parameters of the schemes with respect to the almost sure and weak convergence rates and the
size of the limiting covariance. See e.g.~\cite{DiRe97, DMP08, GaKr74, KoTs04, KuYa93, LaiRob78, LBN94, LBN95, MP11, Pel98b, Pel98a, Pol90, Rup82, Rup91} for results and further references as well as the survey articles and monographs~\cite{BMP90, Duf96, KY03, Lai03, LPW92, Rup91}.
Less attention has been paid to an error control in $L_p$-norm for arbitrary orders $p \ge 2$. We provide such estimates for the Robbins-Monro approximation and the Polyak-Rupert averaging introduced by Ruppert~\cite{Rup91} and Polyak~\cite{Pol90} under mild conditions on the ingredients of these schemes. These estimates build the basis for the error analysis of the multilevel schemes introduced in Section~\ref{sec3}.

Throughout this section we fix $p\in [2,\infty)$, a probability space $(\Omega,\cF,\P)$ equipped with a filtration $(\cF_n)_{n\in\N_0}$,
%  a convex closed domain $\cD\subset \R^d$, 
a scalar product $\langle\cdot,\cdot\rangle$ on $\R^d$ with induced norm $\|\cdot\|$.  Furthermore, we fix a measurable function $f\colon \R^d\to \R^d$ that has a unique zero
$\theta^*\in \R^d$. 

We consider an adapted $\R^d$-valued dynamical system $(\theta_n)_{n\in\N_0}$ iteratively defined by
\begin{equation}\label{dynsys2}
\theta_{n}=\theta_{n-1}+\gamma_n \bigl( f(\theta_{n-1})+ \eps_n\, R_n +\sig_n\, D_{n}\bigr),%\qquad \text{\st{$\text{macro: }\sig_n = \sigma_n/\gamma_n$}}
\end{equation}
for $n\in\N$, where $\theta_0\in \R^d$ is a fixed deterministic starting value, 
\begin{enumerate}
\item[(I)]  $(R_n)_{n\in\N}$ is a previsible process, the \emph{remainder/bias},
\item[(II)] $(D_n)_{n\in\N}$ is a sequence of martingale \emph{differences}, 
\item[(III)] $(\gamma_n)_{n\in\N}$ is a sequence of positive reals tending to zero, and $(\eps_n)_{n\in\N}$ and $(\sig_n)_{n\in\N}$ are sequences of non-negative real numbers.
\end{enumerate}

\subsection{{\bf Estimates for the Robbins-Monro algorithm}}
Our goal is to quantify the speed of convergence of the sequence 
$(\theta_n)_{n\in\N_0}$ to $\theta^*$ in the $p$-th mean sense in terms of 
the \emph{step-sizes} $\gamma_n$, the \emph{bias-levels} $\eps_n$ and the \emph{noise-levels} $\sig_n$.

To this end we employ the following set of assumptions in addition to (I)--(III).
\begin{enumerate}
\item[{\bf A.1}] (Assumptions on  $f$ and $\theta^*$) \label{A1}% \framebox{\st{$ \mu \to 1/L'$}} 
\\[.1cm]
There exist  $L, L'\in (0,\infty)$   such that for all $\theta\in\R^d$
\\[-.2cm]
\begin{itemize}
\item[(i)]  $\scp{\theta-\theta^*}{f(\theta)}\le -L\,\|\theta-\theta^*\|^2$ and\\[-.2cm]
\item[(ii)]  $\scp{\theta-\theta^*}{f(\theta)} \leq -L' \,\|f(\theta)\|^2$.\\[-.1cm]
\end{itemize}
\item[{\bf A.2}]  (Assumptions on $(R_n)_{n\in\N}$ and $(D_n)_{n\in\N}$) \\[.1cm] 
It holds  \\[-.2cm]
\begin{itemize}
\item[(i)]  $\displaystyle{ \sup_{n\in\N}\,\mathrm{esssup} \,  \|R_{n}\| <\infty}$ and\\[-.2cm]
\item[(ii)] $\displaystyle{ \sup_{n\in\N} \E[\|D_{n}\|^p]<\infty}$.\\[-.1cm]
\end{itemize}
\end{enumerate}

\begin{rem}[Discussion of Assumption A.1]\label{r001}
We briefly discuss  A.1(i) and A.1(ii).

Let $\theta\in  \R^d$ and  $c_1,c_2,c_2',\gamma\in (0,\infty)$, and consider the conditions
\begin{align*}
 \langle \theta-\theta^*,f(\theta)\rangle  & \le - c_1\,\|\theta-\theta^*\|^2,\tag{i}\\
 \langle \theta-\theta^*,f(\theta)\rangle & \le - c_2\,\|f(\theta)\|^2,\tag{ii}\\
 \|f(\theta)\| & \le c_2'\, \|\theta-\theta^*\|, \tag{ii'}\\
 \|\theta-\theta^* + \gamma f(\theta)\|^2 & \le \|\theta-\theta^*\|^2 \bigl(1-\gamma c_1 (2-\tfrac{\gamma}{c_2})\bigr). \tag{$\ast$}
\end{align*}
By the Cauchy-Schwartz inequality we have
\begin{equation}\label{aaa0}
f\text{ satisfies (ii)} \, \Rightarrow \, f\text{ satisfies (ii') for every } c_2' \ge 1/c_2,
\end{equation}
and the choice $f(\theta) = \theta$ shows that the 
reverse implication is not valid in general. However, it is easy to check that
\[
f\text{ satisfies (i) and (ii')}\, \Rightarrow \, f\text{ satisfies (ii) for any }c_2\le c_1/(c_2')^2.
\]
Thus, in the presence of condition A.1(i), condition  A.1(ii) is equivalent to a linear growth condition on the function $f$ relative to the zero $\theta^*$.

Finally, conditions (i) and (ii) jointly imply the contraction property (*), which is crucial for the analysis of the Robbins-Monro scheme. We have
\begin{equation}\label{aaa1}
f\text{ satisfies (i) and (ii)}\, \Rightarrow \,  \text{$f$ satisfies ($\ast$) for every  }\gamma \le 2c_2. 
\end{equation}
In fact, let $\gamma \le 2c_2$ and use (ii) and then (i) to conclude that
\begin{align*}
\|\theta-\theta^* + \gamma f(\theta)\|^2 & = \|\theta-\theta^*\|^2 + 2\gamma \langle \theta-\theta^*, f(\theta)\rangle + \gamma^2 \|f(\theta)\|^2 \\ & \le \|\theta-\theta^*\|^2 + \langle \theta-\theta^*, f(\theta)\rangle (2\gamma -\tfrac{\gamma^2}{c_2}) \le \|\theta-\theta^*\|^2 -c_1 \|\theta-\theta^*\|^2 (2\gamma -\tfrac{\gamma^2}{c_2}).
\end{align*}
\end{rem}

In the following we put for $r\in (0,\infty)$ and $n,k\in\N$ with $n \ge k$,
\begin{equation}\label{cen}
\tau_{k,n}(r) = \prod_{j=k+1}^n (1-\gamma_j r),\quad e_{k,n} (r)= \max_{j=k,\dots,n} \eps_j\, \tau_{j,n}(r),\quad s^2_{k,n} (r)= \sum_{j=k}^n \gamma_j^2\,\sig_j^2\, (\tau_{j,n}(r))^2.
\end{equation}

First we provide $p$-th mean error estimates in terms of the quantities  introduced in~\eqref{cen}.

\begin{prop}\label{prop0}
Assume that (I)-(III) and A.1 and A.2 are satisfied. Then for every $r\in (0,L)$ there 
exist $n_0\in\N$ and $\kappa \in (0,\infty)$  such that for all $n\ge k_0 \ge n_0$ we have $\tau_{k_0,n}(r) \in (0,1)$ and  
\begin{equation}\label{esti}
 \E\bigl[\|\theta_n-\theta^*\|^p\bigr]^{1/p}\leq \kappa\, \bigl(\tau_{k_0,n}(r)\,\E[\|\theta_{k_0}-\theta^*\|^p]^{1/p} + e_{k_0,n}(r) + s_{k_0,n}(r)\bigr).
\end{equation}
\end{prop}

\begin{proof}
Without loss of generality we may assume that $\theta^*=0$.

By Assumption A.2 there exists $\kappa_1 \in(0,\infty)$  such that for all $n\in\N$, 
\begin{equation}\label{g1}
\|R_{n}\| \leq \kappa_1 \text{ \ a.s.}
\end{equation}
and
\begin{equation}\label{gg1}
\E[\|D_n\|^p ] \le \kappa_1.
\end{equation}
Note further that  \eqref{aaa0} in Remark~\ref{r001}   implies that the dynamical system \eqref{dynsys2} satisfies $\|\theta_n\|\le (1+\gamma_n/L')\|\theta_{n-1}\| + \gamma_n\eps_n\|R_n\| + \gamma_n\sig_n\|D_n\|$ for every $n\in\N$. With Assumption A.2 we conclude that $\theta_n\in L_p(\Omega,\mathcal F,\P)$ for every $n\in\N$.

Let $r\in (0,L)$. Since $\lim_{n\to\infty} \gamma_n = 0$ we may choose $n_0\in \N$ such that $1-\gamma_n L > 0$ and $1-\sfrac12 \gamma_n/L' \ge (r+L)/(2L) $ for all $n\ge n_0$. Using \eqref{aaa1} in Remark~\ref{r001} %\st{A.1} 
we obtain that for all $\theta\in \R^d$ and for all $n\ge n_0$, 
\begin{equation}\label{e334}
 \begin{aligned}
\|\theta+\gamma_n f(\theta)\|^2
\leq (1 -2\gamma_n L ( 1-\sfrac12 \gamma_n/L')  )\|\theta\|^2 \leq(1- \gamma_n (r+L)/2)^2\, \|\theta\|^2.
\end{aligned}
\end{equation}

In the following we write $\tau_{k,n}$, $e_{k,n}$ and $s_{k,n}$ in place of $\tau_{k,n}(r)$, $e_{k,n}(r)$ and $s_{k,n}(r)$, respectively. Let $k_0 \ge n_0$ and put
\begin{equation}\label{pro1}
\zeta_n= \frac{\theta_n}{\tau_{k_0,n}},\quad 
\xi_n= \frac{\theta_{n-1}+\gamma_n\, (f(\theta_{n-1}) + \eps_n\, R_{n})}{\tau_{k_0,n}},\quad  M_n= \zeta_{k_0}+ \sum_{ k=k_0+1}^n \frac{\gamma_k\,\sig_{k}\, D_{k}}{\tau_{k_0,k}}
\end{equation}
for $n\ge k_0$. Then  $(\zeta_n)_{n\geq  k_0}$ is adapted, $(\xi_n)_{n>  k_0}$ is previsible,  $(M_n)_{n\ge k_0}$ is a martingale and for all  $n>k_0$ we have
\begin{equation}\label{e222}
\zeta_{n}=\xi_n +\Delta M_{n}.
\end{equation}

Below we show that there exists a constant $\kappa_2\in (0,\infty)$, which only depends on $L$, $r$ and $\kappa_1$
such that  a.s.\ for all $n>k_0$, 
\begin{equation}\label{g2}
\|\xi_n\|  \le \|\zeta_{n-1}\| \vee \kappa_2 \,\frac{\eps_n}{\tau_{k_0,n}}
\end{equation}
and
\begin{equation}\label{g3}
 \E\bigl[ [M]_{n}^{p/2}\bigr]^{2/p} \leq \E[\|\theta_{k_0}\|^p]^{2/p} + \kappa_2\,\frac{s^2_{k_0,n} }{\tau^2_{k_0,n}}.
\end{equation}

Observing \eqref{e222} and \eqref{g2}  we may  apply the BDG inequality, see Theorem \ref{thm_BDG}, to the  processes $(\zeta_{n})_{n\ge k_0}$, $(\xi_{n})_{n > k_0}$ and $(M_{n})_{n \ge k_0}$ to obtain for $n\ge k_0$ that 
\begin{align}\label{eq_zeta_est}
\E\bigl[\max_{k_0\le k \le n}\|\zeta_{k}\|^p\bigr]\leq \kappa_3 \,\Bigl( \E\bigl[ [M]_{n}^{p/2}\bigr] +  \bigl(\kappa_2 \frac{e_{k_0,n}}{\tau_{k_0,n}}\bigr)^p\Bigr),
\end{align}
where the constant $\kappa_3>0$ only depends on  $p$. Using \eqref{g3} we  conclude that 
\[
\E\bigl[\|\theta_n\|^p\bigr]
=  \tau_{k_0,n}^{p}\,\E\bigl[\|\zeta_{n}\|^p\bigr]   \le 2^{p/2}\kappa_3\,\bigl(\tau_{k_0,n}^p\,\E[\|\theta_{k_0}\|^p] + \kappa_2^{p/2}s_{k_0,n}^p  +\kappa_2^p\, e_{k_0,n}^{p}\bigr),
\]
which completes the proof of the theorem up to the justification of \eqref{g2} and \eqref{g3}.

For the proof of \eqref{g2} we use \eqref{g1} and \eqref{e334} to obtain that  a.s.\ for $n>k_0$,
\begin{align*}
\|\xi_{n}\| & \le \Bigl\|\frac{\theta_{n-1}+\gamma_n f(\theta_{n-1})}{1-\gamma_n \, r} \frac 1{\tau_{k_0,n-1}}\Bigr\| + \frac{\gamma_n\, \eps_n }{\tau_{k_0,n}}\|R_n\| \\ & \le \frac{1-\gamma_n (r+L)/2}{1-\gamma_n \, r} \|\zeta_{n-1}\| + \kappa_1\frac{\gamma_n \,\eps_n}{\tau_{k_0,n}}
 \le \Bigl(1-\gamma_n\frac{L-r}{2}\Bigr) \|\zeta_{n-1}\| + \kappa_1\frac{\gamma_n\,\eps_n }{\tau_{k_0,n}}, 
\end{align*}
where  the last inequality follows from the fact that $\frac{1-a}{1-b}\leq 1-a+b$ for $0\leq b\leq a\leq1$. 
Hence,  if $ \frac{L-r}{2} \|\zeta_{n-1}\| \ge \kappa_1 \eps_n/\tau_{k_0,n} $ then 
\[
\|\xi_{n}\| \le \|\zeta_{n-1}\|,
\]
 while in the case $ \frac{L-r}{2} \|\zeta_{n-1}\| < \kappa_1\eps_n /\tau_{k_0,n} $, 
\[
\|\xi_{n}\| \le \frac{2\kappa_1}{L-r}\,\frac{ \eps_n}{\tau_{k_0,n}}.
\]
Thus \eqref{g2} holds for any $\kappa_2 \ge 2\kappa_1/(L-r)$.

It remains to show \eqref{g3}. Using \eqref{gg1} we get
\begin{equation}\label{abc6}
\begin{aligned}
\E\bigl[[M]_{n}^{p/2}\bigr]^{2/p} &  =\E\Bigl[\bigl(\|\theta_{k_0}\|^2+ \sum_{k=k_0+1} ^{n} \|\Delta M_k\|^2\bigr)^{p/2}\Bigr]^{2/p}\\
& \leq \E[\|\theta_{k_0}\|^p]^{2/p} +\sum_{k=k_0+1}^{n} \frac{\gamma_k^2\sig_k^2}{\tau_{k_0,k}^2}\bigl(\E[  \|D_k\|^p]\bigr)^{2/p} \\
& \leq \E[\|\theta_{k_0}\|^p]^{2/p} + \kappa_1^{2/p}\sum_{k=k_0+1} ^{n} \frac{\gamma_k^2\,\sig_k^2}{\tau_{k_0,k}^2}  
=  \E[\|\theta_{k_0}\|^p]^{2/p} + \frac{\kappa_1^{2/p}}{\tau_{k_0,n}^2}s_{k_0,n}^2.
\end{aligned}
\end{equation}
Hence \eqref{g3} holds for any $\kappa_2\ge  \kappa_1^{2/p}$,  which completes the proof.
\end{proof}

\begin{rem}  The proof of the $p$-th mean error estimate~\eqref{esti}   in Proposition~\ref{prop0} for the times $n\geq k_0\geq n_0$  makes use of the recursion~(\ref{dynsys2}) for $n$ strictly larger than $k_0$ only. Hence, if $m_0\in\N_0$ and $(\tilde \theta_n)_{n\geq m_0}$ is the dynamical system given by the recursion (\ref{dynsys2}) with an arbitrary random starting value $\tilde \theta_{m_0}\in L_p(\Omega,\mathcal F_{m_0},P)$ then  estimate~\eqref{esti} is  valid  for $\tilde \theta_n$ in place of $\theta_n$ with the same constant $\kappa$ for all $n\ge k_0 \ge \max(n_0,m_0)$.
\end{rem}

The following theorem  provides an estimate for
 the $p$-th mean error of $\theta_n$ in terms of the product
\[
\vel_n=\sqrt{\gamma_n} \,\sig_n.% \qquad \framebox{\st{macro $\vel_n = 1/\sqrt{v_n}$}}
\]
It requires the following additional assumptions  on the step-sizes 
 $\gamma_n$, the bias-levels $\eps_n$ and the noise-levels $\sigma_n$.

\begin{itemize}
\item[{\bf A.3}] (Assumptions on $(\gamma_n)_{n\in\N}$, $(\eps_n)_{n\in\N}$ and $(\sig_n)_{n\in\N}$)\\[.1cm]
We have $\vel_n > 0 $ for all $n\in\N$. 
Furthermore,  with $L$ according to A.1(i), \\[-.1cm]
\begin{itemize}
\item[(i)] $\displaystyle{\limsup_{n\to\infty} \frac{\eps_n}{\vel_n} < \infty}$, and  \\[.1cm]	
\item[(ii)] $\displaystyle{\limsup_{n\to\infty} \frac{1}{\gamma_n}\,\frac{\vel_{n-1} - \vel_{n}}{\vel_{n-1}} <L}$. 
\end{itemize} 
\end{itemize} 

%\framebox{\st{add: convergence results for multilevel are based on the following theorem}} 

\begin{theorem}[Robbins-Monro approximation]\label{thm1} 
Assume that  conditions (I)-(III), A.1, A.2  and A.3 are satisfied. 
Then there exists  $\kappa \in (0,\infty)$  such that for all $n\in \N$,
\[
 \E\bigl[\|\theta_n-\theta^*\|^p\bigr]^{1/p}\leq \kappa\, \vel_n.
\]
\end{theorem}

\begin{proof}
Below we show that there exist $r\in (0,L)$, $\kappa_1\in (0,\infty)$ and $n_1\in\N$ such that for all $n\ge n_1$ we have  $\gamma_n < 1/L$  and
\begin{equation}\label{only}
\tau_{n_1,n}(r) + e_{n_1,n}(r) + s_{n_1,n}(r) \le \kappa_1\, \vel_{n}.
\end{equation}
Then, by choosing $n_0\in \N$  and $\kappa\in (0,\infty)$ according to Proposition~\ref{prop0} and taking $k_0 = \max(n_0,n_1)$ we have for $n\ge k_0$ that $\tau_{k_0,n}(r) = \tau_{n_1,n}(r)/\tau_{n_1,k_0}(r)$, $e_{k_0,n}(r) \le  e_{n_1,n}(r)$ and
$s_{k_0,n}(r) \le  s_{n_1,n}(r)$, and therefore for all $n\geq k_0$
\begin{align*}
 \E\bigl[\|\theta_n-\theta^*\|^p\bigr]^{1/p}\leq \kappa\, 
\bigl(\E[\|\theta_{k_0}-\theta^*\|^p]^{1/p}/ \tau_{n_1,k_0}(r) + 1 \bigr)\, \bigl(\tau_{n_1,n}(r) + e_{n_1,n}(r) + s_{n_1,n}(r)\bigr),
\end{align*}  
which finishes the proof of the theorem, up to the justification of~\eqref{only}.

By Assumption A.3 there exist  $r_1\in (0,L)$, $\kappa_2\in (0,\infty)$ and $n_1\in\N$ such that for all $n>n_1$,
\begin{equation}\label{one}
\frac{\vel_{n-1}}{\vel_n} \le \frac{1}{1-\gamma_n r_1}
\end{equation}
as well as
\begin{equation}\label{one0}
\eps_n \le \kappa_2\, \vel_n.
\end{equation}
Take $r\in (r_1,L)$ and assume without loss of generality that $1-\gamma_n r > 0$ for all  $n\ge n_1$. In the following we write $\tau_{k,n}$, $e_{k,n}$ and $s_{k,n}$ in place of $\tau_{k,n}(r)$, $e_{k,n}(r)$ and $s_{k,n}(r)$, respectively.
It follows from \eqref{one} and $r > r_1$  that the sequence $(\vel_n/\tau_{n_1,n})_{n\ge n_1}$ is increasing and therefore, for all $n\ge n_1$,
\begin{equation}\label{two}
\tau_{n_1,n} = \frac{\tau_{n_1,n}}{\vel_{n_1}}\, \frac{\vel_{n_1}}{\tau_{n_1,n_1}} \le \frac{\vel_n}{\vel_{n_1}}.
\end{equation}
Furthermore, observing \eqref{one0} we also have for all $n\ge n_1$,
\begin{equation}\label{three}
e_{n_1,n} \le  \kappa_2\, \max_{j=n_1,\dots,n} \vel_j\, \tau_{j,n} = \kappa_2\,\tau_{n_1,n}\max_{j=n_1,\dots,n} \frac{\vel_j}{\tau_{n_1,j}} = \kappa_2\, \vel_n.
\end{equation}
Put 
\[
\varphi(n) = \frac{s_{n_1,n}^2}{\vel_n^2}
\]
for $n\ge n_1$. Observing \eqref{one} we obtain  that for $n>n_1$,
\begin{align*}
\varphi(n) & = \frac{\vel_{n-1}^2}{\vel_{n}^2} \, (1-\gamma_n r)^2 \, \varphi(n-1) + \gamma_n \le  \frac{(1-\gamma_n r)^2}{(1-\gamma_{n} r_1)^2}\, \varphi(n-1) + \gamma_n\\
& = \Bigl( 1- \gamma_n \frac{r-r_1}{1-\gamma_n r_1}\Bigr)^2 \, \varphi(n-1) + \gamma_n \le (1- \gamma_n (r-r_1)) \, \varphi(n-1) + \gamma_n.
\end{align*}
This entails that
\[
\varphi(n) - 1/(r-r_1) \le (1- \gamma_n(r-r_1)) (\varphi(n-1) - 1/(r-r_1)),
\]
so that  $\varphi(n) \le \varphi(n-1) \vee 1/(r-r_1)$. Hence, by induction, for all $n\ge n_1$,
\[
\varphi(n) \le  \varphi(n_1)\vee 1/(r-r_1),
\]
so that
\begin{equation}\label{four}
s_{n_1,n} \le (\gamma_{n_1}\vee 1/(r-r_1))^{1/2}\, \vel_n.
\end{equation}
Combining \eqref{two} to \eqref{four} yields \eqref{only}.
\end{proof}

As a particular consequence of Theorem~$\ref{thm1}$ we obtain error estimates in the
case of polynomial step-sizes $\gamma_n$ and noise-levels $\sigma_n$.

\begin{cor}[Polynomial step-sizes and noise-levels]\label{Cor1}
 Assume that conditions (I)-(III), A.1 and A.2 are satisfied and choose $L$ according to A.1(i).  Take $\gamma_0,\sigma_0\in (0,\infty)$, $r_1\in (0,1]$ and $r_2\in \R$ 
 with 
\[
r_1<1 \text{ \ or \ } \bigl(r_1=1\text{ \, and \, }  \gamma_0 > \frac{1+r_2}{2L}\bigr)
\]
and let for all $n\in\N$, 
\[
\gamma_n =  \gamma_0 \frac{1}{n^{r_1}},\quad \sig_n^2 =  \sigma_0^2 \frac{1}{n^{r_2}}.
\]
 Assume further that
\[
\limsup_{n\to\infty} n^{(r_1+r_2)/2}\, \eps_n <\infty.
\]
Then  there exists a constant $\kappa\in(0,\infty)$ such that for all $n\in \N$,
\begin{equation}
\E\bigl[ \|\theta_n-\theta^*\|^p\bigr]^{\frac 1p}\leq  \kappa\,   n^{- \frac{r_1+r_2}{2}}.
\end{equation}
\end{cor}

\begin{proof}
We first verify that Assumption A.3 is satisfied. By definition of $\gamma_n$ and $\sigma_n$ we have
\[
v_n =   \sqrt{\gamma_0}\sigma_0  \frac{1}{n^{(r_1+r_2)/2}}.
\]
Thus, A.3(i) is satisfied due to the assumption on the sequence $(\eps_n)_{n\in\N}$. Moreover, it is easy to see that
\[
\lim_{n\to\infty} \frac{1}{\gamma_n}\,\Bigl(1-\frac{\vel_{n}}{\vel_{n-1}}\Bigr) = \begin{cases}
0, & \text{if }r_1 < 1,\\
\frac{1+r_2}{2  \gamma_0}, & \text{if }r_1 = 1
\end{cases}
\]
and  therefore A.3(ii) is  satisfied as well. Since conditions (I)-(III), A.1 and A.2 are part of the corollary, we may apply Theorem~\ref{thm1} to obtain the claimed error estimate.
\end{proof}

\begin{rem}[Exponential decay of noise-levels]\label{rr11}
Assumption  A.3(ii) may also be satisfied in the case that the noise-levels $\sig_n$ have a superpolynomial decay. For instance, if
\[
\gamma_n = \frac{a_1}{n^{r_1}},\quad \sig_n^2 = \frac{a_2}{n^{r_2}}\,\exp(-a_3 n^{r_3})
\]
for all $n\in\N$, where $a_1,a_2,a_3> 0$, $r_1 > 0$, $r_2\in \R$ and $r_3 \in (0,1)$, then 
\[
\lim_{n\to\infty} \frac{1}{\gamma_n}\,\Bigl(1-\frac{\vel_{n}}{\vel_{n-1}}\Bigr) = \begin{cases} 0,  & \text{if }  r_1 < 1-r_3,\\ \frac{a_3 r_3}{2a_1}, & \text{if } r_1=1-r_3,\\ \infty,  & \text{if }  r_1 > 1-r_3.
\end{cases} 
\]

On the other hand side, if the noise-levels $\sig_n$ are decreasing with exponential decay and the step-sizes $\gamma_n$ are monotonically decreasing then Assumption A.3(ii) is typically not satisfied. In fact, if $\gamma_n \ge \gamma_{n+1}$ for $n\ge n_0$, $\lim_{n\to\infty} \gamma_n = 0$ and $\limsup_{n\to\infty} \sig_{n+1}/\sig_n < 1$ then  $\lim_{n\to\infty} \gamma_n^{-1}\,(1-\vel_{n}/\vel_{n-1} ) = \infty$.

The case of an exponential decay of the noise-levels $\sig_n$ can be treated by applying Proposition~\ref{prop0}. Assume that  conditions (I)-(III), A.1 and A.2  are satisfied.
Assume further that there exist $r\in (0,L)$ and $c\in (0,\infty)$ such that for all $n\in\N$,\\[-.1cm]
\begin{itemize}
	\item[(a)]  $\displaystyle{\sig_n^2 \le c\, \exp(-2rn)}$ and \\[-.1cm]
	\item[(b)] $\displaystyle{\eps_n \le c\, \exp\bigl(-r\sum_{k=1}^n \gamma_k\bigr)}$.
\end{itemize} 
Then there exists  $\kappa \in (0,\infty)$  such that for all $n\in \N$,
\begin{equation}\label{extra}
 \E\bigl[\|\theta_n-\theta^*\|^p\bigr]^{1/p}\leq \kappa \exp\Bigl(-r\sum_{k=1}^n \gamma_k\Bigr).
\end{equation}

\begin{proof}[Proof of~\eqref{extra}]
Since $\lim_{n\to\infty}\gamma_n = 0$ and $1-x\le \exp(-x)$ for all $x\in[0,1]$ we have $(1-\gamma_n r) \le \exp(-r\gamma_n)$ for $n$ sufficiently large. Hence there exists $n_1\in\N$ such that for all $n \ge j\ge n_1$,
\begin{equation}\label{1234}
\tau_{j,n}(r) \le \exp\Bigl(-r\sum_{k=j+1}^n \gamma_k\Bigr).
\end{equation}
Using \eqref{1234} as well as Assumption (b) we get for all $n\ge j\ge n_1$,
\begin{equation}\label{2345}
e_{j,n}(r) \le  (1+c)\, \exp\Bigl(-r\sum_{k=1}^n \gamma_k\Bigr).
\end{equation}
Choosing $n_1$ large enough we may also assume that $\gamma_n \le 1/2$ for all $n\ge n_1$.
Employing \eqref{1234} and Assumption  (a) we then conclude that for all $n\ge j \ge n_1$,
\begin{equation}
\label{345}
\begin{aligned}
	s_{j,n}^2(r) & = \sum_{k=j}^{n} \gamma_k^2\,\sig_k^2\, (\tau_{k,n}(r))^2  \le  \sum_{k=j}^{n} (1+c)  \exp\Bigl(-2rk - 2r\sum_{\ell=k+1}^n \gamma_\ell\Bigr)\\
	& = \exp\Bigl(-2r\sum_{\ell=j+1}^n \gamma_\ell\Bigr)\,\sum_{k=j}^{n}   (1+c) \exp\Bigl(-2rk + 2r\sum_{\ell=j+1}^k \gamma_\ell\Bigr) \\
	& \le \exp\Bigl(-2r\sum_{\ell=j+1}^n \gamma_\ell\Bigr) \,\sum_{k=j}^{n}   (1+c)  \exp(-rk )\le
	 \exp\Bigl(-2r\sum_{\ell=j+1}^n \gamma_\ell\Bigr) \,\frac{  (1+c) }{1-\exp(-r)}.
\end{aligned}
\end{equation}
Combining \eqref{1234} to \eqref{345} with Proposition~\ref{prop0} completes the proof  of~(\ref{extra}).
\end{proof}
\end{rem}

So far we proved error estimates for  the single random variables $\theta_n$.  In the following theorem  we 
establish error estimates, which allow to control the quality of approximation for the whole sequence $(\theta_n)_{n\ge k_0}$ starting from some time $k_0$.

To this end  we   employ the following assumption  A.4, which is stronger  than condition A.3.
\medskip
 
\begin{itemize}
\item[{\bf A.4}] (Assumptions on $(\gamma_n)_{n\in\N}$, $(\eps_n)_{n\in\N}$ and $(\sig_n)_{n\in\N}$)\\[.1cm]
We have $\vel_n > 0 $ for all $n\in\N$. Furthermore,  with $L$ according to A.1(i), there exist  $c_1,c_2,  \eta_1\in (0,\infty)$ as well as $\eta_2 \in (0,1]$ such that $\eta_1 > (1-\eta_2)/p$ and \\[-.2cm]
\begin{itemize}
\item[(i)] $\displaystyle{\limsup_{n\to\infty} \frac{\eps_n}{\vel_n} < \infty}$,  \\[.1cm]
\item[(ii)]$\displaystyle{\limsup_{n\to\infty} \frac{1}{\gamma_n}\,\frac{\vel_{n-1} - \vel_{n}}{\vel_{n-1}} <L}
\text{ and } \vel_n\le \frac{c_1}{ n^{\eta_1}}$ for all but finitely many $n\in\N$, \\[.1cm]
\item[(iii)] $\displaystyle{ \gamma_n \le \frac{c_2}{ n^{\eta_2}}}$ \black
for all but finitely many $n\in\N$.
\end{itemize}
\end{itemize}

\begin{theorem}[Robbins-Monro approximation]\label{thm_max} Assume that conditions (I)-(III), A.1, A.2 and  A.4  are satisfied
	and let 
	\[
	\eta^* =\eta_1 - (1-\eta_2)/p.
	\]
	% alte Definition
	%	\[
	%\eta^* = \begin{cases} \min (c_2 L, \eta_1 - (1-\eta_3)/p), & \text{if }\eta_2 = 1,\\
	%\eta_1 - (1-\eta_3)/p,& \text{if } \eta_2 < 1.\end{cases}
	%\]
	Then for all $\eta \in (0, \eta^*)$
there exists a constant $\kappa\in (0,\infty)$ and $n_0\in \N$ such that for all $k_0\ge n_0$
\begin{align}\label{sup_est}
\E\Bigl[ \sup_{k\geq k_0} k^{p\eta} \,\|\theta_k-\theta^*\|^p\Bigr]^{1/p}\leq  \kappa\,    k_0^{-(\eta^*-\eta)}.
\end{align}
%In particular, for every $\delta>0$, one has}
%\[
%\lim_{n\to\infty} \, n^{\eta^*-\delta}\|\theta_n-\theta^*\| = 0\,\text{ %a.s.}
%\]
\end{theorem}

\begin{proof}
Clearly, we may assume that $\theta^* = 0$. 
Fix $\eta\in (0, \eta^*)$. 

We again use the quantities introduced in~(\ref{cen}). Since Assumption A.4 is stronger than Assumption A.3 we see from the proof of Theorem~\ref{thm1} that there exist $r\in (0,L)$, $\kappa_1 \in (0,\infty)$ and $n_1\in\N$ such that for all $n\ge n_1$ we have $\gamma_n < 1/L$ and 
\begin{equation}\label{tz1}
 \tau_{n_1,n}(r) + e_{n_1,n}(r) +s_{n_1,n}(r)\le \kappa_1 \vel_n,
\end{equation}
cf.~\eqref{only}.
By  A.4(ii) and A.4(iii) we may further assume that for all $n\ge n_1$,
\begin{equation}\label{oo11}
\vel_n \le c_1/n^{\eta_1} \,\,\text{ and }  \gamma_n < \min(1,1/(2r)).  
\end{equation}

Fix $k_0\ge n_1$ and define a strictly increasing sequence $(k_\ell)_{\ell\in\N_0}$ in $\N$ by
\[
k_\ell=\min\Bigl\{m\geq k_0: \sum_{k=k_0+1}^m \gamma_k \geq \ell\Bigr\}.
\]  
Observing the upper bound for $\gamma_n$  in \eqref{oo11} it is then easy to see that for all $\ell\in\N$,
\begin{equation}\label{abc2}
\sum_{k=k_{\ell-1}+1}^{k_\ell} \gamma_k \le 2.
\end{equation}

In the following we write $\tau_{k,n}$, $e_{k,n}$ and $s_{k,n}$ in place of $\tau_{k,n}(r)$, $e_{k,n}(r)$ and $s_{k,n}(r)$, respectively. We  estimate the decay of the sequence $(\tau_{k_0,k_\ell})_{\ell\in\N}$. Let $\ell\in\N$. 
Using \eqref{oo11}, the fact that $1-x \ge \exp(-2x)$ for all $x\in [0,1/2]$, the estimate \eqref{abc2} and the fact that $1-x \le \exp(-x)$ for all $x\in [0,1]$ we get
\begin{equation}\label{abc3}
\tau_{k_0,k_{\ell-1}} = \tau_{k_0,k_{\ell}} \prod_{k=k_{\ell-1}+1}^{k_\ell} (1- \gamma_k r)^{-1} \le \tau_{k_0,k_{\ell}} \prod_{k=k_{\ell-1}+1}^{k_\ell} \exp(2r\gamma_k) \le  \tau_{k_0,k_{\ell}}\exp(4r) 
\end{equation}
as well as
\begin{equation}\label{abc4}
\tau_{k_0,k_{\ell}} \le \prod_{k=k_0+1}^{k_\ell}\exp(-r\gamma_k)  \le \exp(-r\ell).
\end{equation}

Next, we establish a lower bound for the growth of the sequence $(k_\ell)_{\ell\in\N_0}$, namely
\begin{equation}
\label{abc4a}
k_\ell \ge K_\ell
\end{equation}
for all $\ell\in N_0$, where 
\[
K_\ell   =  \begin{cases}
\bigl(\ell(1-\eta_2)/c_2 + k_0^{1-\eta_2}\bigr)^{\frac{1}{1-\eta_2}}, & \text{if }\eta_2 < 1,\\
k_0\exp(\ell/c_2), & \text{if }\eta_2 = 1.
\end{cases}
\]
 In fact, by A.4(iii) we get
\[
\ell \leq \sum_{k=k_0+1}^{k_\ell}\gamma_k \le \sum_{k=k_0+1}^{k_\ell}\frac{c_2}{k^{\eta_2}} \le c_2\int_{k_0}^{k_\ell} x^{-\eta_2}\, dx  = \begin{cases}
\frac{c_2}{1-\eta_2}\bigl(k_ \ell^{1-\eta_2}-k_0^{1-\eta_2}\bigr), & \text{if }\eta_2 < 1,\\
c_2 \ln\bigl( \frac{k_\ell}{k_0}\bigr), & \text{if }\eta_2 =1.\end{cases}
\]
which yields \eqref{abc4a}.

We are ready to establish the claimed  estimate in $p$-th mean  \eqref{sup_est}. Similar to the proof of Proposition~\ref{prop0} we consider the  process $(\zeta_n )_{n\ge n_1}$  and the martingale $(M_n)_{n\ge n_1}$ given by~\eqref{pro1}, where $k_0$ is replaced by $n_1$. As in the proof of Proposition~\ref{prop0} we obtain the maximum estimate in $p$-th mean \eqref{eq_zeta_est} for the process $(\zeta_n)_{n\ge n_1}$ and the estimate in $p/2$-th mean \eqref{abc6} for the quadratic variation $([M]_n)_{n\ge n_1}$. Combining these two estimates we see that for sufficiently large $n_1$ there exists a constant $\kappa_2\in (0,\infty)$,  such that for every $n\ge n_1$ we have
\[
\E\bigl[\max_{n_1\le k \le n}\|\zeta_{k}\|^p\bigr]\leq \kappa_2 \Bigl(\E\bigl[\|\theta_{n_1}\|^{p}\bigr]+\frac{s_{n_1,n}^p + e_{n_1,n}^p}{\tau_{n_1,n}^p}\Bigr).
\]
 Using the latter inequality as well as  \eqref{abc3},  Theorem~\eqref{thm1}  and \eqref{tz1} we may thus conclude that there exists a constant $\kappa_3\in (0,\infty)$, which may depend on $n_1$ but not on $k_0$ such that for every $\ell\in\N$ we have
\begin{align*}
\E\bigl[\max_{k=k_{\ell-1}+1,\dots,k_\ell} \|\theta_k\|^p\bigr] & \le
\tau_{n_1,k_{\ell-1}}^p\, \E\bigl[\max_{k=k_{\ell-1}+1,\dots,k_\ell} \|\zeta_k\|^p\bigr]\\
& \le \kappa_2\,\exp(4rp) \, \tau_{n_1,k_{\ell}}^p\, \Bigl(\E\bigl[\|\theta_{n_1}\|^{p}\bigr]+\frac{s_{n_1,k_\ell}^p + e_{n_1,k_\ell}^p}{\tau_{n_1,k_\ell}^p}\Bigr)
 \\
& \le \kappa_3\,\bigl(\tau_{n_1,k_{\ell}}^p + s_{n_1,k_{\ell}}^p + e_{n_1,k_{\ell}}^p \bigr)\leq  \kappa_3\kappa_1^{ p }\vel_{k_\ell}^{p} .
\end{align*}

Hence, there exists a constant $\kappa_4\in (0,\infty)$ that does not depend on $k_0$  such that 
\begin{equation}\label{abc8}
\E\bigl[\sup_{ k >  k_0}  k^{p\eta} \|\theta_k\|^p\bigr] \leq \sum_{\ell\in\N} \E\bigl[\max_{k=k_{\ell-1}+1,\dots,k_\ell} k_\ell^{p\eta} \|\theta_k\|^p\bigr]  \leq \kappa_4\,\sum_{\ell\in\N} k_\ell^{p\eta}\, \vel_{k_\ell}^{p}.
\end{equation}
Using \eqref{oo11}, the fact that $ p(\eta_1-\eta) > 1-\eta_2$, due to the choice of $\eta$, and the lower bound in \eqref{abc4a} we obtain
\begin{equation}\label{abc10}
\begin{aligned}
\sum_{\ell\in\N} k_\ell^{p\eta}\, \vel_{k_\ell}^{p} & \le c_1^p\,
\sum_{\ell\in\N} k_\ell^{-p(\eta_1-\eta)}  \le c_1^{p}\, \begin{cases}\sum_{\ell\in\N} \bigl( \ell (1-\eta_2)/c_2 + k_0^{1-\eta_2}\bigr)^{-\frac{p(\eta_1-\eta)}{1-\eta_2}}, & \text{ if }\eta_2<1,\\ \sum_{\ell\in\N}\bigl(k_0 \exp(\ell/c_2) \bigr)^{-p(\eta_1-\eta)  }, & \text{ if }\eta_2=1, \end{cases} \\
& \le \kappa_5 \,k_0^{-p(\eta_1-\eta)+1-\eta_2}
\end{aligned}
\end{equation}
with a constant $\kappa_5\in (0,\infty)$ that does not depend on $k_0$.
 Combining \eqref{abc8}  with \eqref{abc10} yields the claimed maximum estimate in $p$-th mean. 
\end{proof}

In analogy to Corollary~$\ref{Cor1}$ we next 
 treat the particular case of
polynomial step-sizes $\gamma_n$  and noise-levels $\sigma_n$.

\begin{cor}[Polynomial step-sizes and noise-levels]\label{Cor2}
 Assume that conditions (I)-(III), A.1 and A.2 are satisfied  and choose  $L$ according to A.1(i). Take  $\gamma_0,\sigma_0\in (0,\infty)$, $r_1\in (0,1]$ and $r_2\in  (-r_1,\infty)$  with 
 \begin{itemize}
 \item[(a)] $r_1<1 \text{ \ or \ } \bigl(r_1=1\text{ \, and \, } \gamma_0> \frac{1+r_2}{2L}\bigr)$,
 \item[(b)] $\frac{r_1+r_2}{2} > \frac{1-r_1}{p}$ 
 \end{itemize}
 and let for all $n\in\N$,
\[
\gamma_n =  \gamma_0\frac{1}{n^{r_1}},\quad \sig_n^2 = \sigma_0^2\frac{1}{n^{r_2}}.
\]
 Assume further that 
\[
\limsup_{n\to\infty} n^{(r_1+r_2)/2}\, \eps_n <\infty.
\]
Then for all $\eta\in (0,  \frac{r_1+r_2}{2} - \frac{1-r_1}{p})$
 there exists a constant $\kappa\in(0,\infty)$ such that for all $k_0\in \N$,
\begin{equation}
\E\bigl[ \sup_{k\geq k_0} k^{p\eta} \,\|\theta_k-\theta^*\|^p\bigr]^{1/p}\leq  \kappa\,  
 k_0^{-\bigl( \frac{r_1+r_2}{2}-\frac{1-r_1}{p}-\eta\bigr)}.
\end{equation}
%In particular, for every $\delta \in (0,\infty)$ we have
%\[
%\lim_{n\to\infty} \,
%n^{\frac{r_1+r_2}{2} - \frac{1-r_1}{p}-\delta}
%\black \|\theta_n-\theta^*\| = 0\,\text{ a.s.}
%\]
\end{cor}

\begin{proof} We first verify Assumption A.4.
By definition of $\gamma_n$ and $\sigma_n$ we have 
\[
v_n =    \sqrt{\gamma_0}\sigma_0 \frac{1}{n^{(r_1+r_2)/2}}.
\]
Thus,  A.4(i) is satisfied due to the assumption on the sequence $(\eps_n)_{n\in\N}$
and the first part of A.4(ii) is satisfied due to Assumption (a), see the proof of Corollary~\ref{Cor1}. Observing  Assumption (b) it is obvious that the second part of A.4(ii)
and Assumption A.4(iii) are satisfied for 
\[
\eta_1 =(r_1+r_2)/2,\quad \eta_2= r_1,\quad c_1=  \sqrt{\gamma_0}\sigma_0,\quad c_2=\gamma_0.
\]

Since conditions (I)-(III), A.1 and A.2 are part of the corollary, we may apply Theorem~\ref{thm_max}  to obtain the claimed error estimate. 
\end{proof}

\subsection{{\bf Estimates for the Polyak-Ruppert algorithm}}

Now we turn to the analysis of Polyak-Ruppert averaging. For $n\in\N$ we let 
\begin{equation}\label{average}
\bar\theta_n=\frac 1{\bar b_n}\sum_{k=1}^n 
%\1_{\{\max_{j=k\dots,n-1}|\theta_j- \theta_n|\leq C\}}\,
b_k\,\theta_k,
\end{equation}
where 
$(b_k)_{k\in\N}$ is a fixed sequence of  strictly positive reals and 
\[
\bar b_n=\sum_{k=1}^n b_k.
\]
We estimate  the speed of convergence  of $(\bar \theta_n )_{n\in\N}$ to $\theta^*$ in  $p$-th mean in terms of the sequence  $(\bar v_n)_{n\in\N}$  given by 
\[
\bar v_n = \frac{\vel_n}{\sqrt{n\,\gamma_n}} = \frac{\sig_n}{\sqrt{n}}.
\]
To this end we will replace the set of assumptions A.1, A.2 and A.3 by the following 
set of assumptions B.1, B.2 and B.3. Note that B.2 coincides with A.2 while B.1 is stronger
than A.1 and B.3 is stronger than  A.3, see Remark~$\ref{r003}$ below.

\begin{enumerate}
\item[{\bf B.1}] (Assumptions on  $f$ and $\theta^*$)\label{B1}  \\[.1cm]
There exist  $L, L',L'',\lambda \in (0,\infty)$ and a matrix $H\in\R^{d\times d}$   such that for all $\theta\in\R^d$  \\[-.2cm]
\begin{itemize}
\item[(i)]  $\scp{\theta-\theta^*}{f(\theta)}\le -L\,\|\theta-\theta^*\|^2$,\\[-.2cm]
\item[(ii)] $\scp{\theta-\theta^*}{f(\theta)} \leq -L' \,\|f(\theta)\|^2$  and\\[-.1cm]
\item[(iii)]  $\|f(\theta)- H (\theta-\theta^*)\|\le L''\,\|\theta-\theta^*\|^{1+\lambda}$.\\[-.1cm]
\end{itemize}
\item[{\bf B.2}]  (Assumptions on $(R_n)_{n\in\N}$ and $(D_n)_{n\in\N}$) \\[.1cm] 
It holds  \\[-.2cm]
\begin{itemize}
\item[(i)]  $\displaystyle{ \sup_{n\in\N}\,\mathrm{esssup} \,  \|R_{n}\| <\infty}$ and\\[-.2cm]
\item[(ii)] $\displaystyle{ \sup_{n\in\N} \E[\|D_{n}\|^p]<\infty}$.\\[-.1cm]
\end{itemize}
\item[{\bf B.3}]  (Assumptions on $(\gamma_n)_{n\in\N}$, $(\eps_n)_{n\in\N}$, $(\sig_n)_{n\in\N}$ and $(b_n)_{n\in\N}$) \\[.1cm]
 We have $\sigma_n>0$ for all $n\in\N$. The sequence $(\gamma_n)_{n\in\N}$ is decreasing and the sequences  $(n\gamma_n)_{n\in\N}$ and $(b_n\sig_n)_{n\in\N}$ are increasing. Moreover, with $L$ and $\lambda$ according to B.1 there exist  $\nu,c_1,c_2,c_3\in [0,\infty)$ with $c_2>(\nu+1)/L$ 
% and  $\eta\in (0,1)$ 
such that\\[-.2cm]
\begin{itemize}
\item[(i)] $\displaystyle{\limsup_{n\to\infty} \frac{\eps_n}{\bar v_n} < \infty}$,\\[-.1cm]
\item[(ii)] $\displaystyle{\limsup_{n\to\infty} \frac{1}{\gamma_n}\,\frac{\bar v_{n-1}-\bar v_n  }{\bar v_{n-1}} <L}$  and  $\displaystyle{\vel_n^{1+\lambda} \le c_1\bar v_n }$ for all but finitely many $n\in\N$, \\[-.1cm]
\item[(iii)] %$\displaystyle{\frac {c_2}{n}\le \gamma_n\le \frac{c_3}{n^\eta} }$
$\displaystyle{ \gamma_n\ge \frac {c_2}{n}}$ for all but finitely many $n\in \N$,\\[-.1cm]
\item[(iv)] $\displaystyle{b_m\leq c_3\, b_n \Bigr(\frac mn\Bigr)^{\nu}}$ for all $m\ge n\ge 1$,\\ i.e., the sequence $(b_n)_{n\in\N}$ has at most polynomial growth. 
\end{itemize}
\end{enumerate}
%\end{assu}

\begin{rem}[Discussion of assumptions B.1 and  B.3]\label{r003}
We first show that  Assumption B.3 implies Assumption A.3. Since $(n\gamma_n)_{n\in\N}$ is increasing we have $\eps_n/ v_n = \eps_n/(\sqrt{n\gamma_n} \bar v_n) \le \eps_n/(\sqrt{\gamma_1}\, \bar v_n)$ for every $n\in\N$, which proves that B.3 implies A.3(i). Furthermore, 
\[
\frac{v_n-v_{n+1}}{v_n}=\frac {\bar v_n- \tfrac{\sqrt{(n+1)\gamma_{n+1}}}{\sqrt{n\gamma_n}}\bar v_{n+1}}{\bar v_n}\le \frac{\bar v_n-\bar v_{n+1}}{\bar v_n}
\]
for every $n\in\N$, which proves that B.3 implies A.3(ii).

We add that, due to the presence of Assumption B.1(ii), it is sufficient to require that $f$ satisfies the inequality in B.1(iii) on some open  ball around $\theta^*$.  In fact, let $D\subset \R^d$ and $\delta\in (0,\infty)$ be such that  $B(\theta^*,\delta) = \{\theta\in \R^d\colon \|\theta-\theta^*\|< \delta\}\subset D$. \black Let $c_2,c_3,c'_3,\lambda\in (0,\infty)$ and $H\in \R^{d\times d}$ and consider the conditions
\begin{align*}
\forall \theta\in D\colon & \langle \theta-\theta^*,f(\theta)\rangle \le - c_2\,\|f(\theta)\|^2,\tag{ii}\\
\forall \theta\in D\colon &  \|f(\theta) - H(\theta-\theta^*)\| \le c_3\|\theta-\theta^*\|^{1+\lambda}, \tag{iii}\\
\forall \theta\in  B(\theta^*,\delta) \colon & \|f(\theta) - H(\theta-\theta^*)\| \le c'_3\|\theta-\theta^*\|^{1+\lambda}. \tag{iii'}
\end{align*}
Then 
\begin{equation}\label{aaa3}
f\text{ satisfies (ii) and (iii')}\, \Rightarrow \,  \text{$\|H\| \le 1/c_2$ and $f$ satisfies (iii) for every  }c_3\ge \max(c'_3, \tfrac{2}{c_2\delta^\lambda}), 
\end{equation}
where $\|H\|$ denotes the induced matrix norm of $H$. 

For a proof of~$\eqref{aaa3}$ we first note that (iii') implies that $H (\theta) = \lim_{0 < \varepsilon \to 0} \tfrac{1}{\varepsilon} f(\theta^*+\varepsilon \theta)$ for every $\theta\in \R^d$. Using~$\eqref{aaa0}$  we conclude that $\|H\| \le 1/c_2$. For $\theta\in D\setminus  B(\theta^*,\delta)$ we have $\|\theta-\theta^*\| \ge \delta$. Observing the latter fact and using~$\eqref{aaa0}$ again we conclude
\[
\|f(\theta)- H(\theta-\theta^*)\| \le \|f(\theta)\| + \| H(\theta-\theta^*)\|\le \tfrac{2}{c_2}\|\theta-\theta^*\| \le \tfrac{2}{c_2\delta^\lambda}\|\theta-\theta^*\|^{1+\lambda}.
\]

As an immediate consequence of~\eqref{aaa3} with the choice $\delta\le 1$ we obtain that if $f$ satisfies B.1(ii),(iii) then $f$ satisfies B.1(iii) for  every $\lambda'\in [0,\lambda]$ 
with $L''$ replaced by $\max(L'',2/(L'\delta^{\lambda'}))$.

\end{rem}

\begin{theorem}[Polyak-Ruppert approximation]\label{thm2} Assume that conditions (I)-(III) and  B.1-B.3 are satisfied.
Put $q=\frac p{1+\lambda}$ with $\lambda$ according to B.1. Then there exists $\kappa\in (0,\infty)$  such that the Polyak-Ruppert algorithm \eqref{average} 
satisfies for all $n\in\N$
\[
 \E\bigl[\|\bar\theta_n-\theta^*\|^q\bigr]^{1/q}\leq \kappa\, \bar v_n.
\]
\end{theorem}

For the proof of Theorem \ref{thm2} we follow the approach of the  classical paper~\cite{Pol90} by first comparing the dynamical system $(\theta_n)_{n\ge 0}$ with a linearised version $(y_n)_{n\geq 0}$ given by $y_{0}=\theta_{0}$ and 
\[
y_{n}=y_{n-1}+ \gamma_n \bigl( H (y_{n-1}-\theta^*)+ \sig_n D_{n}\bigr)
\]
for $n\in\N$.

\begin{lemma}\label{le0412-1}
Assume that conditions (I)-(III) and   B.1-B.3 are satisfied. Put 
$q=\frac p{1+\lambda}$ with $\lambda$ according to B.1. Then there exists $\kappa\in (0,\infty)$  such that for all $n\in\N$
\[
\E\bigl[\|\theta_n-y_n\|^q\bigr]^{1/q}\leq \kappa\,\bar v_n.
\]
\end{lemma}

\begin{proof} 
Without loss of generality we may assume that $\theta^*=0$.

Using B.3(i),(ii) we see that there exist $r\in (0,L)$, $n_0\in\N$ and $\kappa_0\in (0,\infty)$ such that for all $n\ge n_0$ we have 
\begin{equation}\label{eeq1}
\eps_n\le  \kappa_0\, \bar v_n
\end{equation}
and 
\begin{equation}\label{eeq2}
\frac{\bar v_{n-1}}{\bar v_{n}} \le  \frac{1}{1-\gamma_n r}.
\end{equation}
Since $\lim_{n\to\infty}\gamma_n = 0$ we may assume that $\gamma_n \le 1/(2r)$ for all $n\ge n_0$.

By Remark~\ref{r003} conditions A.1-A.3 are satisfied and we  may apply  Theorem~\ref{thm1} to obtain the existence of $\kappa_1\in (0,\infty)$ such that for all $n\in \N$,
\begin{equation}\label{eeq3}
\E[\|\theta_n\|^p]^{1/p} \le  \kappa_1\, \vel_n.
\end{equation}
 Furthermore,  estimate~\eqref{e334}  in the proof of Proposition~\ref{prop0} is valid, i.e., there exists $n_1\in\N$ such that for all $n\ge n_1$ and all $\theta\in\R^d$,
\begin{equation}\label{eeq4}
\|\theta + \gamma_n f(\theta)\| \le (1- \gamma_n (r+L)/2)\|\theta\|.
\end{equation}

By Assumption B.1(iii) we have 
\begin{equation}\label{this0}
H\theta = \lim_{\eps \downarrow 0} \eps^{-1} f(\eps \theta)
\end{equation}
for every $\theta\in \R^d$. Using~\eqref{eeq4} we may therefore conclude that for all $n\ge n_1$ and all $\theta\in\R^d$,
\begin{equation}\label{eeq5}
\|\theta + \gamma_n H \theta\| \le (1- \gamma_n (r+L)/2)\|\theta\|.
\end{equation}

Let $n_2= \max(n_0,n_1)$. For  $n\geq n_2$ we put
\[
z_n=\theta_n-y_n \,\text{ and }\, \delta_n = \frac{\E\bigl[\|z_n\|^q\bigr]^{1/q}}{\bar v_n}.
\] 
Let $n>n_2$. Using \eqref{eeq5}, Assumptions B.1(iii), B.2(i) and \eqref{eeq1} we  see that there exists $\kappa_2\in (0,\infty)$ such that 
\begin{align*}
\|z_n\| &  = \|z_{n-1} + \gamma_{n} \bigl(H z_{n-1}+ f(\theta_{n-1})- H\theta_{n-1} +\eps_n\, R_{n}\bigr)\|\\
& \le \|z_{n-1} + \gamma_{n} H z_{n-1}\| + \gamma_{n} \|f(\theta_{n-1})- H\theta_{n-1}\| + \gamma_{n}\,\eps_n\, \|R_{n}\|\\
& \le (1-\gamma_n (r+L)/2)\|z_{n-1}\| + \gamma_{n}\, L''\, \|\theta_{n-1}\|^{1+\lambda} + \kappa_2\, \gamma_{n}\,\bar v_n\quad \text{a.s.}, 
\end{align*}
and  employing  \eqref{eeq2}, \eqref{eeq3},  B.3(ii) and the fact that  $\gamma_n\leq 1/(2r)$  we conclude that
\begin{equation}\label{x2}
\begin{aligned}
\delta_n & \le  (1-\gamma_n (r+L)/2)\,\frac{ \bar v_{n-1}}{ \bar v_n}\, \delta_{n-1} + L''\,\gamma_n\frac{\bar v_{n-1}}{\bar v_n}\, \frac{\E\bigl[ \|\theta_{n-1}\|^p\bigr]^{1/q}}{\bar v_{n-1}} +\kappa_2\gamma_n\\
& \le \frac{1-\gamma_n (r+L)/2}{1-\gamma_n\, r}\, \delta_{n-1} +\frac{L''\,\kappa^{1+\lambda}_1\, c_1}{1-\gamma_n\, r}\, \gamma_n+ \kappa_2\gamma_n\\
& \le (1-\gamma_n (L-r)/2)\,\delta_{n-1} + \bigl(2L''\,\kappa^{1+\lambda}_1\, c_1 + \kappa_2\bigr)\,\gamma_n. 
\end{aligned}
\end{equation}
Put $\kappa_3 = (L-r)/2 >0$ and $\kappa_4= 2L''\,\kappa^{1+\lambda}_1\, c_1 + \kappa_2$. By~\eqref{x2} we have for $n\ge n_2$ that $
\delta_n \le (1-\kappa_3\, \gamma_n )\,\delta_{n-1} +\kappa_4\, \gamma_n$
or, equivalently, 
\[
\frac {\delta_n}{\kappa_4}-\frac 1{\kappa_3} \le  (1-\kappa_3\,\gamma_n)\Bigl(\frac {\delta_{n-1}}{\kappa_4}-\frac 1{\kappa_3}\Bigr),
\]
which yields,
\[
\frac {\delta_n}{\kappa_4}-\frac 1{\kappa_3} \le \bigl(\frac {\delta_{n_2}}{\kappa_4}-\frac 1{\kappa_3} \bigr)\exp\Bigl(- \kappa_3\sum_{k= n_2+1}^n \gamma_k\Bigr).
\]
Since $\sum_{k\in\N}\gamma_k = \infty$, due to  Assumption B.3(iii), we conclude that
\[
\limsup_{n\to\infty} \, (\delta_n/ \kappa_4-1/\kappa_3) \leq  0,
\]
which finishes the proof.
\end{proof}

\begin{proof}[Proof of Theorem \ref{thm2}] 
Without loss of generality we may assume that $\theta^*=0$.

For all $n\in\N$ we have
\begin{equation}\label{z45v}
\bar \theta_n = \frac{1}{\bar b_n}\sum_{k=1}^n b_k (\theta_k-y_k)  + \frac{1}{\bar b_n}\sum_{k=1}^{n} b_k y_k.
\end{equation}
We separately analyze the two terms on the right hand side of~\eqref{z45v}.

By Assumption B.3(iv) it follows that 
\begin{equation}\label{z1}
\bar b_n \geq \frac {b_n}{c_3} \sum_{k=1}^n\Bigl(\frac kn\Bigr)^\nu\geq \kappa_1 n b_n,
\end{equation}
where $\kappa_1=(c_3(\nu+1))^{-1}$.

Employing Lemma~\ref{le0412-1} as well as~\eqref{z1} and the fact that $(b_k\sig_k)_{k\in\N}$ is increasing we see that there exists $\kappa_2\in (0,\infty)$ such that for all $n\in\N$,
\begin{equation}\label{z45}
\E\Bigl[ \Bigl\|\frac{1}{\bar b_n}\sum_{k=1}^n b_k (\theta_k-y_k)\Bigr\|^q\Bigr]^{1/q}\le \frac{\kappa_2 }{n\, b_n}\sum_{k=1}^n b_k  \bar v_k = \frac{\kappa_2}{n \,b_n} \,\sum_{k=1}^n \frac{b_k \,\sig_k}{\sqrt{k}}  \le  \frac{\kappa_2\,\sig_n}{n } \,\sum_{k=1}^n \frac{1}{\sqrt{k}}\le 2\kappa_2\,\bar v_n,
\end{equation}
where we used  $\sum_{k=1}^n 1/\sqrt k\leq 2 \sqrt n$ in the latter step.

Next, put $b_0=0$ and let
\[ 
 \Upsilon_{k,n}=\prod_{\ell=k+1}^{n}(I_d+\gamma_\ell H),
 \]
 where $I_d\in\R^{d\times d}$ is the identity matrix, 
 as well as
 \[
\bar \Upsilon_{k,n} = \sum_{m=k}^n b_m \Upsilon_{k,m} 
 \]
for $0\leq k\leq n$. Then, for all $n\in \N$,
\[
y_n= \Upsilon_{0,n}\, \theta_0+\sum_{k=1}^n \gamma_k\,\sig_{k}\,\Upsilon_{k,n} \, D_k
\] 
and
\[
\sum_{k=1}^n b_k \, y_k =\bar \Upsilon_{0,n}\, \theta_0 + \sum_{k=1}^n \gamma_k\, \sig_{k} \bar \Upsilon_{k,n} \, D_k.
\]
Using the Burkholder-Davis-Gundy inequality we obtain that there exists a constant $\kappa_3\in (0,\infty)$ such that for all $n\in\N$,
\begin{align}\label{eq1}
\E\biggl[\biggl\|\sum_{k=1}^n b_k\, y_k \biggr\|^q\biggr]^{1/q} \leq \kappa_3\, \biggl( \| \bar\Upsilon_{0,n}\|\, \|\theta_0\|+ \E\Bigl[ \Bigl(\sum_{k=1}^n  \gamma_k^2\,\sig_{k}^2 \,\|\bar \Upsilon_{k,n}\|^2 \|D_k\|^2\Bigr)^{q/2}\Bigr]^{1/q}\biggr).
\end{align}
By Assumption B.2(ii)  there exists a constant $\kappa_4\in (0,\infty)$  such that for all $n\in\N$,
\begin{equation}\label{noch}
\begin{aligned}
\E\Bigl[ \Bigl(\sum_{k=1}^n \gamma_k^2\,\sig_{k}^2\, \|\bar \Upsilon_{k,n}\|^2\, \|D_k\|^2\Bigr)^{q/2}\Bigr]^{1/q}&\leq  \Bigl(\sum_{k=1}^n \gamma_k^2\,\sig_k^2\, \|\bar \Upsilon_{k,n}\|^2 \,\E\bigl[  \|D_k\|^q\bigr]^{2/q}\Bigr)^{1/2}\\
&\leq \kappa_4\,  \Bigl(\sum_{k=1}^n \gamma_k^2\,\sig_{k}^2 \,\|\bar \Upsilon_{k,n}\|^2\Bigr)^{1/2}.
\end{aligned}
\end{equation}

We proceed with estimating the norms $\|\bar \Upsilon_{k,n}\|$. Since $c_2 > (\nu+1)/L$ we can fix $r\in ((\nu+1)/c_2,L)$ and  proceed  as in the proof of  Lemma~\ref{le0412-1}   to conclude that there exists $n_1\in\N$ such that for all $n\ge n_1$
\[
\|I_d+\gamma_n H\| \le 1- \gamma_n(r+L)/2,
\]
see \eqref{eeq5}.
The latter fact and the assumption that the sequence $ (n\gamma_n)_{n\in\N}$ is increasing  jointly imply that for $n\ge k\geq  n_1-1$
\begin{align*}
\|\Upsilon_{k,n}\| & \leq  \prod_{\ell=k+1}^{n} \Bigl(1-\frac{\gamma_\ell\,(r+L)}{2}
\Bigr)\leq  \prod _{\ell=k+1}^{n} \Bigl(1- \frac{\gamma_k\,k\,(r+L)}{2\ell}\Bigr) \\ & \le \exp\Bigl(-\frac{(r+L)\gamma_k k}{2}\sum_{\ell=k+1}^n \frac{1}{\ell}\Bigr)\le \Bigl(\frac{k+1}{n+1}\Bigr)^{(r+L) \gamma_k k/2},
\end{align*}
where we used that $1-z\leq e^{-z}$ for all $z\in\R$.
 Employing the latter estimate as well as Assumption B.3(iv) we get that for $n\ge k\ge n_1-1$
\begin{equation}\label{tre1}
\begin{aligned}
\|\bar \Upsilon_{k,n}\|  & \leq  \sum_{\ell=k}^n b_\ell\,\|\Upsilon_{k,\ell}\|\leq c_3\, b_k
 \sum_{\ell=k}^n\Bigl(\frac{\ell}{k}\Bigr)^{\nu}\Bigl(\frac{k+1}{\ell+1 }\Bigr)^{(r+L) \gamma_k k/2} \\ & \le c_3\, b_k\, 2^{\nu}
 \sum_{\ell=k}^n\Bigl(\frac{k+1}{\ell+1}\Bigr)^{(r+L) \gamma_k k/2-\nu}.
\end{aligned}
\end{equation}
Put $\beta_k=(r+L) \gamma_k k/2-\nu$ and note that 
by the choice of $r$ and by B.3(iii) one has for $k$ large enough that
\[
\beta_k= (L-r) \gamma_k k/2 + r \gamma_k k-\nu > (L-r) \gamma_k k/2  + 1.
\]
Choosing $n_1$ large enough we therefore conclude that there exists $\kappa_5\in (0,\infty)$ such that for  $n\geq k\ge n_1-1$,
\[
 \sum_{\ell=k}^n\Bigl(\frac{k+1}{\ell+1}\Bigr)^{(r+L) \gamma_k k/2-\nu} \leq 1+ (k+1)^{\beta_k} \int_{k+1}^{\infty} t^{-\beta_k}\, dt=1+ \frac{k+1}{\beta_k-1}\leq  \frac{ \kappa_5}{\gamma_k}.
\]
In combination with \eqref{tre1}  we see that there  exists $\kappa_6\in (0,\infty)$ such that for all $n\geq k\ge n_1-1$,
\begin{equation}\label{also}
\|\bar \Upsilon_{k,n}\|\leq \kappa_6 \,\frac{b_k}{\gamma_k}.
\end{equation}

For $0\le k < n_1-1 \le n$ we have
\[
\bar \Upsilon_{k,n} = \sum_{\ell=k}^{n_1-2} b_\ell\, \Upsilon_{k,\ell} + \bar\Upsilon_{n_1-1,n} \, \Upsilon_{k,n_1-1}
\]
and, observing \eqref{also}, we may thus conclude that for $0\le k < n_1-1 \le n$,
\begin{equation}\label{v1aa}
\|\bar \Upsilon_{k,n}\| \le \Bigl(\max_{0\le j\le \ell\le n_1-1}\| \Upsilon_{j,\ell}\|\Bigr) \,
\Bigl(\sum_{\ell=1}^{n_1-2} b_\ell + \kappa_6 \,\frac{b_{n_1-1}}{\gamma_{n_1-1}}\Bigr).
 \end{equation}

Using \eqref{also} as well as \eqref{v1aa} and the fact that the sequence $(b_n\sig_n)_{n\in\N}$ is increasing we conclude that there exists $\kappa_7\in (0,\infty)$ such that for all $n\in\N$,
\begin{equation}\label{v1} 
\biggl(\sum_{k=1}^n \gamma_k^2\,\sig_{k}^2 \, \|\bar \Upsilon_{k,n}\|^2\biggr)^{1/2} \le \kappa_7 \, \sqrt{n}\, b_n \, \sig_{n}.
\end{equation}

Combining \eqref{eq1}, \eqref{noch} and \eqref{v1}, employing again that the sequence $(b_n\sig_n)_{n\in\N}$ is increasing and observing~\eqref{z1}   we see that there exists a constant $\kappa_8\in (0,\infty)$ such that for all $n\in\N$ , 
\begin{equation}\label{eq137}
\E\biggl[\biggl\|\frac{1}{\bar b_n}\sum_{k= 1}^n b_k\, y_k \biggr\|^q\biggr]^{1/q} \leq \frac{\kappa_8 }{\bar b_n}\, \sqrt{n}\, b_n \, \sig_{n}  \le  \frac {\kappa_8}{\kappa_1}\, \bar v_n
\end{equation}

Combining \eqref{z45} with \eqref{eq137} completes the proof of the theorem.
\end{proof}

We  consider  the particular case of polynomial  step-sizes $\gamma_n$, noise-levels $\sigma_n$ and weights $b_n$.

\begin{cor}[Polynomial step-sizes, noise-levels and weights]\label{Cor3}
Assume that conditions (I)-(III), B.1 and B.2 are satisfied  and  let $q\in[\frac p{1+\lambda},p)$ with $\lambda$ according to B.1(iii).

  Take $\gamma_0,\sigma_0,b_0  \in (0,\infty)$, 
   $r_1\in (0,1)$,  $r_2\in (-r_1,\infty)$ and $r_3\in[  r_2/2,\infty)$ with 
\[
\frac {1+r_2}{r_1+r_2}\le \frac pq
\]
and let for all $n\in\N$, 
\[
\gamma_n =\gamma_0\frac{1}{n^{r_1}},\quad \sig_n^2 = \sigma_0^2\frac{1}{n^{r_2}}, \quad b_n=  b_0 n^{r_3}.
\]
Assume further that
\[
\limsup_{n\to\infty} n^{(1+r_2)/2}\, \eps_n <\infty
\]

Then there exists a constant $\kappa\in(0,\infty)$ such that for all $n\in\N$,
\[
\E[\|\bar \theta_n-\theta^*\|^{q}]^{1/q} \leq  \kappa\, n^{-\frac 12(r_2+1)}.
\]
\end{cor}

\begin{proof}
Conditions (I)-(III), B.1 and B.2 are part of the corollary. Further note that  by Remark~\ref{r003}  condition B.1(iii) remains  true when replacing $\lambda$ by $\lambda'=\frac pq-1\in(0,\lambda]$. We verify that condition B.3 holds as well with $\lambda'$ in place of $\lambda$. Then the corollary is a consequence of Theorem~\ref{thm2} and the fact that $\bar v_n=\sigma_n/\sqrt n  = \sigma_0  n^{-\frac 12 (r_2+1)}$.

 Since $r_1\in (0,1)$ it is clear that $(\gamma_n)_{n\in\N}$ is decreasing and $(n\gamma_n)_{n\in\N}$ is increasing. Furthermore,
\[
(b_n\sigma_n)_{n\in\N}= ( \sigma_0 b_0 n^{r_3-r_2/2})_{n\in\N}
\] is increasing since $r_3\geq r_2/2$.
B.3(i)  is satisfied due to the  assumption on $(\eps_n)_{n\in\N}$.  Since $r_1<1$ the limes superior  in B.3(ii) is zero. Moreover,  the second estimate of B.3(ii) holds for an appropriate  positive constant  $c_1$ since  $v_n^{1+\lambda'} = (\sqrt{\gamma_n}\sigma_n)^{1+\lambda'} = ( \gamma_0\sigma_0^2 )^{\frac{p}{2q}} n^{-\frac{p}{q} \frac{r_1+r_2}{2}}$ and \black $\frac pq \frac{r_1+r_2}2\geq \frac {r_2+1}2$ by assumption.  Condition B.3(iii) is satisfied  for any $c_2\in (0,\infty)$ since $r_1 < 1$, and 
thus condition B.3(iv) is satisfied with $\nu= r_3$.  
\end{proof}

\section{Multilevel stochastic approximation}\label{sec3}

Throughout this section we fix 
$p\in [2,\infty)$, 
a probability space $(\Omega,\cF,\P)$,  a scalar product $\langle\cdot,\cdot\rangle$ on $\R^d$ with induced norm $\|\cdot\|$, 
a non-empty set $\cU$ equipped with some $\sigma$-field, a random variable $U\colon \Omega\to \cU$ and a product-measurable function 
\[
F\colon \R^d\times \cU \to \R^d,
\]
such that $F(\theta,U)$ is integrable for every $\theta\in\R^d$.  We consider the function $f\colon\R^d\to\R^d$ given by
\begin{equation}\label{mlf}
f(\theta)= \E[F(\theta, U)]
\end{equation}
and we assume that $f$ has a unique zero $\theta^*\in \R^d$. 

Our goal is to compute $\theta^*$ by means of  stochastic approximation algorithms based on  the multilevel Monte Carlo approach. To this end we suppose that we are given a hierarchical scheme 
\[
F_1,F_2,\ldots\colon \R^d\times \cU\to \R^d
\]
of suitable product-measurable approximations to $F$, such that $F_k(\theta,U)$ is integrable and $F_k(\theta,U) - F_{k-1}(\theta,U)$ can be simulated for all $\theta\in\R^d$ and $k\in\N$, where $F_0=0$.

To each random vector $F_k(\theta,U)-F_{k-1}(\theta,U)$ we assign  a positive number  $C_k\in (0,\infty)$, which  depends only on the level $k$ and may represent a deterministic worst case upper bound of the  average computational cost or  average runtime needed to compute a single simulation
 of $F_k(\theta,U)-F_{k-1}(\theta,U)$. 
As  announced in the introduction
 we impose assumptions  on the approximations $F_k$ and the cost bounds $C_k$ that are similar in spirit to the  classical multilevel Monte Carlo setting, see~\cite{Gil08}.

\begin{enumerate}
\item[{\bf C.1}] (Assumptions on $(F_k)_{k\in\N}$ and $(C_k)_{k\in\N}$)\label{C1}\\[.1cm]
There exist measurable functions $\Gamma_1,\Gamma_2\colon \R^d\to (0,\infty)$  and constants $M\in(1,\infty)$ and $K, \alpha,\beta\in(0,\infty)$ with $\alpha \ge \beta$ such that for all $k\in\N$ and all $\theta\in\R^d$ \\[-.4cm]
\begin{itemize}
\item[(i)]
 $\E[\|F_{k}(\theta,U)-F_{k-1}(\theta,U)-\E[F_{k}(\theta,U)-F_{k-1}(\theta,U)]\|^p]^{1/p}\le  \Gamma_1(\theta)\,M^{-k\beta}$, \\[-.4cm]
\item[(ii)] $\|\E[F_{k}(\theta,U)-F(\theta,U)]\|\le \Gamma_2(\theta)\, M^{-k\alpha}$, and\\[-.4cm]
\item[(iii)] $ C_k \le K M^{k}$.
\end{itemize}
\end{enumerate}

 We combine the Robbins-Monro algorithm with  the classical multilevel approach taken in \cite{Gil08}. The proposed method uses in each Robbins-Monro step a multilevel estimate with  a complexity that is adapted to the actual state of the system and   increases in time. 
 
The algorithm is 
specified by  the parameters $\Gamma_1, \Gamma_2, M, \alpha, \beta$ from Assumption C.1, an initial vector $\theta_0\in\R^d$, 
\begin{itemize}
\item[(i)] a sequence of step-sizes $(\gamma_n)_{n\in\N}\subset (0,\infty)$ tending to 
zero,\\[-.4cm]   
\item[(ii)] a  sequence of bias-levels  $( \eps_n \black)_{n\in\N}\subset  (0,\infty)\black$, and \\[-.4cm]
\item[(iii)] a  sequence of noise-levels $( \sigma_n \black)_{n\in\N}\subset  (0,\infty)$.
\black
\end{itemize}
 The maximal level $m_n(\theta)$ and the number of iterations $N_{n,k}(\theta)$ on level $k\in \{1,\dots,m_n(\theta)\}$ that are used by the  multilevel estimator in the $n$-th Robbins-Monro step depend on $\theta\in\R^d$ and are determined in the following way.
 We take  
\begin{equation}\label{maxlevel}
m_n(\theta) =  1 \vee \Bigl\lceil \frac{1}{\alpha}\,\log_M\Bigl(\frac{\Gamma_2(\theta)}{ \eps_n }\Bigr)\Bigr\rceil\in\N,
\end{equation}
i.e. $m_n(\theta)$ is the smallest $m\in\N$ such that $\Gamma_2(\theta) M^{-\alpha m} \leq  \eps_n$ holds true for the bias bound in Assumption C.1(ii).  Furthermore, 
\begin{equation}\label{itnum}
N_{n,k}(\theta)=\bigl\lceil   \kappa_n(\theta)\,  M^{-k(\beta+1/2)}\bigr\rceil,
\end{equation}
where
\begin{equation}\label{par33}
\kappa_n(\theta)=\begin{cases}(\Gamma_1(\theta)/\sigma_n)^2 \,M^{m_n(\theta)(\frac 12-\beta)_+}, & \text{if }\beta\neq 1/2,\\
 (\Gamma_1(\theta)/\sigma_n)^2 \,m_n(\theta), & \text{if }\beta = 1/2.
 \end{cases}
\end{equation}

Take a sequence $(U_{n,k,\ell})_{n,k,\ell\in \N }$ of independent copies of $U$.
We use 
\begin{equation}\label{mul1}
Z_n(\theta) =\sum_{k=1}^{m_n(\theta)} \frac{1}{N_{n,k}(\theta)} \sum_{\ell=1}^{N_{n,k}(\theta)} \bigl(F_k(\theta, U_{n,k,\ell})- F_{k-1}(\theta,U_{n,k,\ell})\bigr)
\end{equation}
as a multilevel approximation of $f(\theta)$ in
the $n$-th Robbins-Monro step, and we study the sequence of Robbins-Monro approximations
 $(\theta_n)_{n\in\N_0}$ given by
\begin{equation}\label{rm10}
\theta_n =\theta_{n-1}+\gamma_n\, Z_n(\theta_{n-1}).
\end{equation}

We measure the computational cost of  $\theta_n$ by the quantity 
 \begin{equation}\label{cost1}
 \cost_n  =\E\Bigl[\sum_{j=1}^n \sum_{k=1}^{m_j(\theta_{j-1})} N_{j,k}(\theta_{j-1}) \,C_k\Bigr].
 \end{equation}
 That means we take the  mean computational cost for simulating the random vectors  $F_k(\theta,U_{j,k,\ell})-F_{k-1}(\theta,U_{j,k,\ell})$ for the first $n$ iterations into account and we ignore the cost of the involved arithmetical operations. Note, however, that the number of arithmetical operations needed to compute $\theta_n$ is essentially proportional to $\sum_{j=1}^n \sum_{k=1}^{m_j( \theta_{j-1})} N_{j,k}( \theta_{j-1})$, and the  average of the latter quantity is captured by $\cost_n$ under the weak assumption that $\inf_k C_k >0$.
 
 Note further that the  quantity $\cost_n$ depends
 on the parameters  $\Gamma_1,\Gamma_2, M,\alpha,\beta,\theta_0,(\gamma_k)_{k=1,\dots,n}$, $ (\eps_k )_{k=1,\dots,n}$ and   $(\sigma_k )_{k=1,\dots,n}$,  which determine the algorithm $(\theta_k)_{k\in\N}$ up to time $n$. For ease  of notation we  do not explicitly indicate this dependence  in the notation $\cost_n$.

 To obtain upper  bounds of  $\cost_n$ we need the following additional 
 assumption  C.2  on the functions $\Gamma_1,\Gamma_2$, which implies that both the variance estimate in C.1(i) and the bias estimate in C.1(ii) are at most of polynomial growth in $\theta\in\R^d$ with exponents related to the parameters $\alpha$, $\beta$ and $p$.
 
 \begin{enumerate}
 	\item[{\bf C.2}] (Assumption on $\Gamma_1,\Gamma_2$)\label{C2}\\[.1cm]
 	With $\alpha,\beta,\Gamma_1,\Gamma_2$ according to Assumption C.1 there exists $K_1\in (0,\infty)$ and
 	\[
 	\beta_1 \in \begin{cases} [ 0, \min(\beta,1/2)], & \text{if }\beta \neq 1/2,\\  [ 0, 1/2), & \text{if }\beta = 1/2,\end{cases}
 	\]
 	such that for all $\theta\in \R^d$
 	\begin{equation}\label{condgamma}
 	\Gamma_1(\theta) \le K_1\, (1+  \| \theta\|)^{\beta_1 p} \text{ and } \Gamma_2(\theta) \le K_1\, (1+ \|\theta\|)^{\alpha p}.
 	\end{equation}
 \end{enumerate}

We are now in the position to state the central complexity theorem  on the multilevel Robbins-Monro algorithm.

\begin{theorem}[Multilevel Robbins-Monro approximation]\label{thm_Gilesnew}
Suppose that Assumption A.1 is satisfied for the function $f$ given by~\eqref{mlf}  and that Assumptions  C.1 and C.2 are satisfied. Take $L\in (0,\infty)$ according to A.1, take $\Gamma_1, \Gamma_2, M, \alpha, \beta$ according to C.1 and C.2 and let $\theta_0\in\R^d$.
 
Take $r\in (-1,\infty)$, $\gamma_0\in (\frac{1+r}{2L} ,\infty)$, $\sigma_0,\eps_0\in (0,\infty)$,
 and let $\rho= \frac 12 (1+r)$ and for all $n\in\N$,
\[
\gamma_n = \gamma_0 \frac{1}{n},\quad \sigma_n^2 = \sigma_0^2 \frac{1}{n^r},\quad \eps_n = \eps_0 \frac{1}{n^\rho}.
\]
Then for all $\eta\in (0, \rho)$ there exists  $\kappa \in (0,\infty)$  such that 
for all $n\in\N$, 
\[
  \E\bigl[\sup_{k\ge n} k^{\eta p}\,\|\theta_k-\theta^*\|^p]^{1/p}\le \kappa\,n^{-(\rho-\eta)}. \]
In particular, for all $\delta \in (0,\infty)$ we have $\lim_{n\to\infty} n^{\rho-\delta}\, \|\theta_n-\theta^*\| = 0$ almost surely.

If additionally $\alpha > \beta \wedge 1/2$ and $r > \frac{\beta \wedge 1/2}{\alpha -  \beta \wedge 1/2}$ then there exists $\kappa'\in (0,\infty)$ such that for all $n\in\N$,
\[
\cost_n\le \kappa'\,\begin{cases}
 n^{2\rho}, &\text{ if }\beta >1/2,\\
 n^{2\rho} \,(\ln (n+1))^2, &\text{ if }\beta=1/2,\\
 n^{2\rho\, \bigl(1 + \tfrac{1-2\beta}{2\alpha}\bigr)}, &\text{ if }\beta < 1/2.
 \end{cases}
 \]
\end{theorem}

The implementation of the multilevel Robbins-Monro approximation from Theorem \ref{thm_Gilesnew} requires the knowledge of a positive lower bound for the  parameter $L$ from Assumption A.1.
This difficulty is overcome by applying the Polyak-Ruppert averaging methodology. That means we consider  the approximations 
\begin{equation}\label{pr10}
\bar\theta_n=\frac 1{\bar b_n}\sum_{k=1}^n b_k\,\theta_k,
\end{equation}
were  $(\theta_n)_{n\in\N}$ is the multilevel Robbins-Monro scheme  specified by~\eqref{rm10}, $(b_k)_{k\in\N}$ is a sequence of  positive reals and 
\[
\bar b_n=\sum_{k=1}^n b_k
\]
for $n\in\N$, see Section \ref{sec2}.

Note that the cost to compute $\bar \theta_n$ differs from the cost to compute $\theta_n$ at most by a deterministic factor, which does not depend on $n$. Therefore we again  measure the computational cost for the computation of $\bar \theta_n$ by the quantity $\cost_n$ given by \eqref{cost1}. 

We state the second complexity theorem, which concerns   Polyak-Ruppert averaging.
\begin{theorem}[Multilevel Polyak-Ruppert approximation]\label{thm_Gilesnew_2}
Suppose that Assumption B.1 is satisfied for the function $f$ given by~\eqref{mlf}  and that Assumptions  C.1 and C.2 are satisfied. Take $\lambda\in (0,\infty)$ according to B.1, take $\Gamma_1, \Gamma_2, M, \alpha, \beta$ according to C.1 and C.2 and let $\theta_0\in\R^d$. 

Let $q\in[\frac p{1+\lambda},p)$. Take  $\gamma_0, \sigma_0, \eps_0, b_0\in (0,\infty)$,  $r_1\in(0,1)$, $r_2\in(-r_1,\infty)$ and $r_3\in [r_2/2,\infty)$ with
\[
\frac{1+r_2}{r_1+r_2} \leq \frac pq,
\]
and let  $\rho=\frac 12 (1+r_2)$ and for all $n\in\N$,
\[
\gamma_n= \gamma_0 \frac 1{n^{r_1}},\quad \eps_n= \eps_0  \frac 1{n^\rho},\quad \sigma^2_n= \sigma_0^2 \frac 1{n^{r_2}}, \quad b_n= b_0 \, n^{r_3}.
\]
Then there exists $\kappa\in(0,\infty)$ such that for all $n\in\N$
$$
\E\bigl[\|\bar \theta_n-\theta^*\|^q]^{1/q}\le \kappa\,n^{-\rho}.
$$
 If additionally  $\alpha > \beta \wedge 1/2$ and $r_2 \ge \frac{\beta \wedge 1/2}{\alpha -\beta \wedge 1/2}$ 
%$$
%r_2\geq \begin{cases} \frac1{2\alpha-1}, & \beta\geq 1/2\\
%\frac \beta{\alpha-\beta}, & \beta<1/2\end{cases} \ \text{ or, equivalently, } \ \rho\geq %\begin{cases} \frac\alpha{2\alpha-1}, & \beta\geq 1/2\\
%\frac 12\frac \alpha{\alpha-\beta}, & \beta<1/2\end{cases},
%$$
 then there exists $\kappa'\in(0,\infty)$ such that for all $n\in\N$
\[
\cost_n\le \kappa'\,\begin{cases}
 n^{2\rho}, &\text{ if }\beta >1/2,\\
 n^{2\rho} \,(\ln (n+1))^2, &\text{ if }\beta=1/2,\\
 n^{2\rho\, \bigl(1 + \tfrac{1-2\beta}{2\alpha}\bigr)}, &\text{ if }\beta < 1/2.
\end{cases}
\]
\end{theorem}

\begin{rem}\label{comp}
Assume the setting of Theorem~\ref{thm_Gilesnew} or~\ref{thm_Gilesnew_2} and let $e_n  =\E[\|\theta_n -\theta^*\|^p]^{1/p}$ or $e_n =\E[\| \bar\theta_n\black-\theta^*\|^{ q}\black ]^{1/ q}$, respectively.
Then there exists $\kappa\in (0,\infty)$ such that for every $n\in\N$,
	\begin{equation}\label{ce}
	\cost_n \le \kappa \begin{cases}e_n^{-2}, & \text{if } \beta>1/2,\\
	e_n^{-2}(\ln(1+e_n^{-1}))^{ 2},  & \text{if }\beta=1/2,\\
	e_n^{-2-\frac {1-2\beta}{2\alpha}}, &\text{if } \beta<1/2.
	\end{cases}
	\end{equation}
Note that these bounds for the computational cost in terms of the  error coincide with the respective bounds for the multilevel computation of a single expectation presented  in~\cite{Gil08}.
\end{rem}

\black
\begin{rem}\label{general}
The multilevel stochastic approximation algorithms analysed in Theorems~\ref{thm_Gilesnew} and~\ref{thm_Gilesnew_2} are based on evaluations of the increments $F_k-F_{k-1}$. Consider, more generally, a sequence of measurable mappings $P_k\colon\R^d\times \cU\to\R^d$, $k\in\N$, such that for all $k\in\N$,
\[
\E[P_k(\theta,U)]=\E[F_k(\theta,U)-F_{k-1}(\theta,U)]
\] 
and $C_k$ is a worst case cost bound for simulating $P_k(\theta,U)$.
Then  Theorems~\ref{thm_Gilesnew} and~\ref{thm_Gilesnew_2} are still valid for the algorithm obtained by using $P_k$ as a substitute for the increment $F_k-F_{k-1}$ in \eqref{mul1} if Assumption C.1(i) is satisfied with $P_k$ in place of $F_k-F_{k-1}$. 
\end{rem}

The proofs of Theorems~\ref{thm_Gilesnew} and~\ref{thm_Gilesnew_2} are based on the following proposition, which  shows that under Assumptions C.1(i),(ii) the scheme~\eqref{rm10} can be represented as a Robbins-Monro scheme of the general form~\eqref{dynsys2}  studied  in Section~\ref{sec2}.
%As a consequence, error estimates for the scheme~\eqref{rm10} and its Polyak-Ruppert averaging %can be directly inferred from the respective error estimates in Section~\ref{sec2}.
 It  further  provides an estimate of the  computational cost~\eqref{cost1} based on Assumptions C.1(iii) and C.2 only.

\begin{prop}\label{prop1}\mbox{} 
\begin{itemize}
\item[(i)]Suppose that Assumptions C.1(i),(ii)  are satisfied.
Let $\cF_n$ denote the $\sigma$-field generated by the variables $U_{m,k,\ell}$ with $m,k,\ell\in\N$ and $m\le n$, and let $\cF_0$ denote the trivial $\sigma$-field. The scheme $(\theta_n)_{n\in\N_0}$ given by~\eqref{rm10} satisfies 
\[
\theta_n = \theta_{n-1}+\gamma_n \bigl( f(\theta_{n-1})+ \eps_n\, R_n +\sig_n\, D_{n}\bigr)
\]
for every $n\in\N$, where $(R_n)_{n\in\N}$ is a previsible process with respect to the filtration $(\cF_n)_{n\in\N_0}$ and $ (D_n)_{n\in\N}$ is a sequence of martingale differences with respect to  $(\cF_n)_{n\in\N_0}$ and  $(R_n)_{n\in\N}$ and $ (D_n)_{n\in\N}$ satisfy Assumption A.2.
\item[(ii)] Suppose that  Assumptions  C.1(iii)  and C.2 are satisfied. Then
there exists a constant $\kappa \in (0,\infty)$ such that for all $n\in\N$ the  computational cost~\eqref{cost1}  of $\theta_n$ given by \eqref{rm10} satisfies
 \begin{equation}\label{costest}
 \cost_n \le \kappa \,\max_{k=1,\dots,n-1} \E[(1+\|\theta_k\|)^p]\,\begin{cases}\sum_{j=1}^n \bigl(\eps_j^{-1/\alpha} + 
 \sigma_j^{-2}\, \eps_j^{- \frac{(1-2\beta)_+} {\alpha}}\bigr), & \text{if }\beta \neq  1/2,\\[.2cm]
 \sum_{j=1}^n \bigl(\eps_j^{-1/\alpha} + \sigma_j^{-2}\, (\log_M(1/\eps_j))^{2}\bigr),& \text{if }\beta = 1/2.
 \end{cases}
 \end{equation}
\end{itemize}
\end{prop}

\begin{proof}
 We first prove statement (i) of the proposition. 
Put $f_n(\theta) = \E[F_n(\theta,U)]$ and let
\[
P_{n,k,\ell}(\theta) = F_k(\theta, U_{n,k,\ell})- F_{k-1}(\theta,U_{n,k,\ell})
\]
for $n,k,\ell\in \N$ and $\theta\in \R^d$. 
By Assumptions C.1(i),(ii)  we have 
\begin{equation}\label{j1}
\E[\|P_{n,k,\ell}(\theta) -\E[P_{n,k,\ell}(\theta)] \|^p]^{1/p} \le \Gamma_1(\theta) \,M^{-k\beta} \ \text{ and } \  \|f_k(\theta)- f(\theta)\|\le  \Gamma_2(\theta)\, M^{-k\alpha}
\end{equation}
for all $n,k,\ell\in \N$ and $\theta\in \R^d$. 

By \eqref{j1} and the definition~\eqref{maxlevel} of $m_n(\theta)$ we get  for all $n\in\N$ and $\theta\in\R^d$ that
\begin{align}\label{eq9283}
\|\E[Z_n(\theta)]-f(\theta)\|= \|\E[F_{m_n(\theta)}(\theta,U)]-f(\theta)\|\leq  \Gamma_2(\theta)\, M^{-m_n(\theta) \alpha} \leq \eps_n. 
\end{align}
Furthermore, by the Burkholder-Davis-Gundy inequality, the triangle inequality on the $L^{p/2}$-space, $\eqref{j1}$ and the definition $\eqref{itnum}$ of $N_{n,k}(\theta)$ there exists $c_1\in (0,\infty)$,
which only depends on $p$, such that
\begin{align*}
& \E[\|Z_n(\theta)-\E[Z_n(\theta)]\|^p]^{2/p} \\
& \qquad \qquad = \E\Bigl[ \Bigl\| \sum_{k=1}^{m_n(\theta)} \sum_{\ell=1}^{N_{n,k}(\theta)} \frac 1{N_{n,k}(\theta)}\bigl(  P_{n,k,\ell}(\theta)-\E[P_{n,k,\ell}(\theta)]\bigr)\Bigr\|^p\Bigr]^{2/p}\\
& \qquad \qquad\leq  c_1\, \E\Bigl[ \Bigl( \sum_{k=1}^{m_n(\theta)} \sum_{\ell=1}^{N_{n,k}(\theta)}  \frac 1{N_{n,k}(\theta)^2}\bigl\|  P_{n,k,\ell}(\theta)-\E[P_{n,k,\ell}(\theta)]\bigr\|^2\Bigr)^{p/2}\Bigr]^{2/p}\\
&\qquad \qquad\leq c_1\,\sum_{k=1}^{m_n(\theta)}  \frac 1{N_{n,k}(\theta)^2} \sum_{\ell=1}^{N_{n,k}(\theta)}  \E\bigl[\bigl\|  P_{n,k,\ell}(\theta)-\E[P_{n,k,\ell}(\theta)]\bigr\|^p\bigr]^{2/p}\\
&\qquad \qquad\leq  c_1  \, \Gamma_1(\theta)^2\sum_{k=1}^{m_n(\theta)}  \frac 1{N_{n,k}(\theta)} \,  M^{-2\beta k}  \leq  c_1 \, \frac{\Gamma_1(\theta)^2}{\kappa_n(\theta)}\sum_{k=1}^{m_n(\theta)}  M^{k(1/2-\beta)}. 
\end{align*}
Recalling the definition of $\kappa_{n}(\theta)$, see \eqref{par33}, we conclude that 
there exists $c_2\in (0,\infty)$ such that for all $n\in\N$ and $\theta\in\R^d$
\begin{equation}\label{eq32989}
\E[\|Z_n(\theta)-\E[Z_n(\theta)]\|^p]^{2/p} \le c_2\, \sigma_n^2.
\end{equation}

With
\begin{align*}
&R_n  :=\frac{1}{ \eps_n}\,\E[Z_n(\theta_{n-1})- f(\theta_{n-1})|\cF_{n-1}]\text{ \ and}\\
&D_n :=\frac{1}{\sigma_n}\, \bigl( Z_n(\theta_{n-1})- f(\theta_{n-1})  -\E[ Z_n(\theta_{n-1})- f(\theta_{n-1})|\cF_{n-1}]\bigr)
\end{align*}
we obtain that $\theta_n= \theta_{n-1} + \gamma_n(f( \theta_{n-1}) +  \eps_n\, R_n + \sig_n D_n)$. We  verify  Assumption A.2.

The  process $(R_n)_{n\in\N}$  is predictable and using the independence  of $(U_{n,k,\ell})_{k,\ell\in\N}$ and $\cF_{n-1}$  we conclude with  \eqref{eq9283} that  $\sup_{n\in\N}\|R_n\| \le 1$.  By the latter  independence it further  follows that $(D_n)_{n\in\N}$ is a sequence of martingale differences, which satisfies  $\sup_{n\in\N}\E[\|D_n\|^p] \le c_2^{p/2}$ as a consequence of \eqref{eq32989}.  This completes the proof of statement (i). 

We turn to the proof of statement (ii). 
Let $j\in\N$ and $\theta\in\R^d$. Using Assumption C.1(iii),  we conclude that there exists $c\in (0,\infty)$, which only depends on $K$, $M$ and $\beta$ such that
\begin{equation}\label{fff1}
\begin{aligned}
\sum_{k=1}^{m_j(\theta)} N_{j,k}(\theta)\, C_k & \le \sum_{k=1}^{m_j(\theta)} \bigl(1+\kappa_j(\theta)\, M^{-k(\beta+ 1/2)}\bigr) K M^k \\ & \le \frac{K}{1-M^{-1}} M^{m_j(\theta)} +  K\kappa_j(\theta) \,
\begin{cases}
\frac{M^{m_j(\theta)(1/2-\beta)_+}}{1-M^{-|1/2-\beta|}}, & \text{if }\beta \neq 1/2,\\
m_j(\theta), & \text{if }\beta = 1/2
\end{cases}\\
& \le c\, M^{m_j(\theta)} +  c\, \frac{\Gamma_1^2(\theta)}{\sigma_j^2} \begin{cases}
M^{m_j(\theta)(1-2\beta)_+}, & \text{if }\beta \neq 1/2,\\
(m_j(\theta))^2, & \text{if }\beta = 1/2.
\end{cases}
\end{aligned}
\end{equation}
Furthermore,  \eqref{maxlevel} yields that
\begin{equation}\label{mm1}
m_j(\theta) \le \alpha^{-1} (\log_M(\Gamma_2(\theta)) + \log_M(\eps_j^{-1})) +1\text{ \ and \ }M^{m_j(\theta)} \le M\, \eps_j^{-1/\alpha} (\Gamma_2(\theta))^{1/\alpha}.
\end{equation}
Combining \eqref{fff1} with \eqref{mm1} and employing Assumption C.2 we see that there exists $c_1\in (0,\infty)$, which only depends on $K$, $K_1$,  $M$, $\beta$ and $\alpha$, such that 
\begin{equation}\label{fgf1}
\begin{aligned}
\sum_{k=1}^{m_j(\theta)} N_{j,k}(\theta)\, C_k  & \le c_1 \eps_j^{-1/\alpha} (1+\|\theta\|)^p \\
& \quad + c_1 \sigma_j^{-2} (1+\|\theta\|)^{2\beta_1 p}
\begin{cases}
\eps_j^{-(1-2\beta)_+/\alpha} (1+ \|\theta\|)^{(1-2\beta)_+ p} , & \text{if }\beta \neq 1/2,\\
(\log_M(\eps_j^{-1}(1+\|\theta\|)^{\alpha p}))^2, & \text{if }\beta = 1/2.
\end{cases}
\end{aligned}
\end{equation}
Suppose that  $\beta\not=1/2$. Then~\eqref{fgf1} implies
\[
\sum_{k=1}^{m_j(\theta)} N_{j,k}(\theta)\, C_k   \le c_1  (1+\|\theta\|)^p  \bigl(\eps_j^{-1/\alpha} +\sigma_j^{-2} \eps_j^{-\frac {(1-2\beta)_+}\alpha}\bigr),
\]
which finishes the proof for the case $\beta\not=1/2$. In the case $\beta=1/2$ we have $\beta_1 < 1/2$ and therefore 
the existence of a constant $c_2\in (0,\infty)$, which does not depend on $\theta$, such that 
\[
(\log_M (1+ \|\theta\|^{\alpha p}))^2 \le c_2 \, (1+ \|\theta\|)^{p(1-2\beta_1)}. 
\]
One completes the proof of statement (ii)  by combining the latter estimate with \eqref{fgf1}.	
	\end{proof}

We turn to the proof of Theorem~\ref{thm_Gilesnew}.

\begin{proof}[Proof of Theorem~\ref{thm_Gilesnew}] 
The error estimate follows by Corollary~\ref{Cor2} since Assumption A.1 is part of the theorem and Assumptions (I)-(III) and A.2 are satisfied by Proposition~\ref{prop1}(i). 

It remains to prove the cost estimate.  The error estimate implies that $\sup_{n\in\N} \E[(1+\|\theta_n\|)^p] <  \infty$. Employing  Proposition~\ref{prop1}(ii) \black we thus see that there exists $c_1\in (0,\infty)$ such that for every $n\in\N$
\black
\begin{equation}\label{uu2}
\cost_n  \le c_1 \begin{cases}
\sum_{j=1}^n \bigl(j^{\rho/\alpha} + 
 j^{2\rho(1+(1/2-\beta)_+/\alpha)-1}\bigr), & \text{if }\beta \neq  1/2,\\[.2cm]
\sum_{j=1}^n \bigl(j^{\rho/\alpha} + \rho^2\, j^{2\rho-1}(\log_M(j))^{2}\bigr),& \text{if }\beta = 1/2.
\end{cases}
\end{equation}
Hence there exists $c_2\in (0,\infty)$ such that for every $n\in\N$
\begin{equation}\label{uu2a}
\cost_n 
 \le  c_2 \begin{cases}
n^{\rho/\alpha+1} + n^{2\rho(1+(1/2-\beta)_+/\alpha)}, & \text{if }\beta \neq  1/2,\\[.2cm]
n^{\rho/\alpha+1} + n^{2\rho}(\log_M(n))^{2},& \text{if }\beta = 1/2.
\end{cases}
\end{equation}
If $\beta\ge 1/2$ then  $\alpha > 1/2$ and $r>\frac{1}{2\alpha-1}$, which implies that $\rho/\alpha < r = 2\rho-1$ and therefore 
\[
n^{\rho/\alpha+1} \le n^{2\rho} = n^{2\rho(1+(1/2-\beta)_+/\alpha)}.
\]
 If $\beta < 1/2$ then  $\alpha > \beta$ and $r>\frac{\beta}{\alpha-\beta}$, which implies that $2\rho \frac{\alpha-\beta}{\alpha} > 1$ and therefore
 \[
 n^{\rho/\alpha+1} \le  n^{\rho/\alpha+2\rho \frac{\alpha-\beta}{\alpha}} = n^{2\rho(1+(1/2-\beta)_+/\alpha)}.
 \]
This completes the proof. 
\end{proof}

We proceed with the proof of Theorem~\ref{thm_Gilesnew_2}.

%\begin{rem}
%For fixed $q\in [\frac p{1+\lambda},p)$ and $r_1\in(0,1)$ the choice
%$$
%r_2=\frac {q-pr_1}{p-q}=-r_1+\frac{1(1-r_1)}{p-q}>-r_1
%$$
%leads to the minimal value for $r_2$ and, respectively, for $\rho$ for which the error estimate of the theorem is still applicable. 
%\end{rem}

\begin{proof}[Proof of Theorem~\ref{thm_Gilesnew_2}]
The error estimate follows with Corollary~\ref{Cor3} since Assumption B.1  is part of the theorem  and Assumptions  (I)-(III) and  B.2 hold  by Proposition~\ref{prop1}(i).
The cost estimate in the theorem is proved in the same way as the cost estimate in  Theorem~\ref{thm_Gilesnew}.  One only observes that $\sup_{n\in\N} \E[(1+\|\theta_n\|)^p] <  \infty$ is valid since the assumptions in Corollary~\ref{Cor3} are stronger than the assumptions in Corollary~\ref{Cor1}.
\end{proof}

\section{General convex closed domains}\label{sec4}

In this section we  extend
the results of  Sections~\ref{sec2} and~\ref{sec3} 
to convex domains. In the following $D$ denotes a convex and closed subset  of $\R^d$ and $f\colon D\to\R^d$ is a function with a unique zero $\theta^*\in D$.  We start with the Robbins-Monro scheme.

 Let
\[
\mathrm{pr}_D\colon \R^d\to D
\]
denote the orthogonal projection on $D$ with respect to the given inner product $\langle\cdot,\cdot\rangle$ on $\R^d$ and define the dynamical system $(\theta_n)_{n\in\N_0}$ by the recursion
\begin{equation}\label{projdyn}
\theta_{n}= \mathrm{pr}_D\bigl(\theta_{n-1}+\gamma_n ( f(\theta_{n-1})+ \eps_n\, R_n +\sig_n\, D_{n})\bigr)
\end{equation}
in place of~\eqref{dynsys2}, where $\theta_0\in D$ is a deterministic starting value in $D$.
Then the following fact follows by a straightforward  modification in the proofs of Proposition~\ref{prop0}
and Theorem~\ref{thm_max} using the contraction property of $\mathrm{pr}_D$.

\begin{ext}\label{Ext1}
Proposition~\ref{prop0}, Theorem~\ref{thm1}, Corollary~\ref{Cor1}, Theorem~\ref{thm_max},  Corollary~\ref{Cor2} and the statement on the system~\eqref{dynsys2} in Remark~\ref{rr11} remain valid for the system~\eqref{projdyn} in place of~\eqref{dynsys2} if $\R^d$ is replaced by $D$  in Assumption A.1.
\end{ext}

Analogously, we  extend Theorem~\ref{thm_Gilesnew} in Section~\ref{sec3} on the multilevel Robbins-Monro approximation   to the case where the mappings $F, F_1, F_2,\dots $ are defined on $D\times\mathcal U$ with $D$ being a closed and convex subset of $\R^d$ and 
\[
f\colon D\to \R^d,\quad \theta \mapsto \E[F(\theta,U)]
\]
has a unique zero $\theta^*\in D$. In this case  we proceed analogously to Extension~\ref{Ext1} and  employ the  projected multilevel Robbins-Monro scheme
\begin{equation}\label{rmproj}
\theta_n = \pr_D\bigl(\theta_{n-1} + \gamma_n\, Z_n(\theta_{n-1})\bigr) 
\end{equation}
with $\theta_0\in D$ and $Z_n$ given by~\eqref{mul1}, in place of the multilevel scheme~\eqref{rm10}. 

Note  that if $\pr_D$ can be evaluated on $\R^d$ with constant cost then, up to a constant depending  on $D$ only, the computational cost of the projected approximation $\theta_n$  is still bounded by the quantity $\cost_n$ given by~\eqref{uu2} since the computation of $\theta_n$ requires $n$ evaluations of $\pr_D$ and $\cost_n \ge C_1 n$.

Employing Proposition~\ref{prop1}  as well as Extension~\ref{Ext1} one easily gets  the following result.

\begin{ext}\label{Ext3}
Theorem~\ref{thm_Gilesnew} remains valid for the scheme~\eqref{rmproj} in place of~\eqref{rm10} if  $\R^d$ is replaced by $D$ in Assumptions A.1, C.1 and C.2.
\end{ext}

Next we consider the Polyak-Ruppert scheme. In this case we additionally suppose that $D$ contains an open ball  $B(\theta^*,\delta)=\{\theta\in\R^d\colon\|\theta-\theta^*\|<\delta\}$ around the unique zero $\theta^*\in D$ and we \black extend the function $f$ on $\R^d$: for $c\in (0,\infty)$ define
 \begin{equation}\label{continue}
  f_{c}\colon \R^d \to \R^d, \quad x\mapsto  - c (x-\mathrm{pr}_D (x)) + f(\mathrm{pr}_D(x)).
\end{equation}
The following lemma shows that property B.1 is preserved for appropriately chosen $c>0$.  

\begin{lemma}\label{l234} Let $\delta>0$ and suppose that $B(\theta^*,\delta) \subset D$
 and that $f\colon D\to \R^d$ satisfies B.1(i) to B.1(iii) on $D$. Take $L,L',L'',\lambda,H$ according to B.1,  let $ c\in (1/(2L'),\infty) $  and  put  
 \[
r_c = 1-\tfrac{L}{c} \bigl(2 - \tfrac{1}{ cL'}) \in [0,1).
\]
Then 
$f_c$ satisfies B.1(i) to B.1(iii) on $\R^d$ with  
\begin{equation}\label{newpar}
 L_c = c\bigl(1-\sqrt{r_c}\bigr),\,\,  L_c' =  \frac{ L_c}{\frac2{(L')^2}+2c^2},\,\, L_c''= \bigl(c+\tfrac{1}{L'}\bigr)\tfrac{1}{\delta^\lambda} + L''
\end{equation}
in place of $L$, $L'$ and $L''$, respectively.
\end{lemma}

\begin{proof}
 Using~\eqref{aaa1} with $c_1 = L$, $c_2= L'$ and $\gamma= 1/c$ it follows that $r_c\in [0, 1)$.
By~\eqref{aaa1} and the contractivity of the  projection $\mathrm{pr}_D$ we have for any $\theta\in\R^d$ that
\begin{align*}
\langle \theta-\theta^*, f_{c}(\theta)\rangle & =  \langle\theta-\theta^*, - c \,(\theta- \mathrm{pr}_D(\theta)) + f(\mathrm{pr}_D(\theta))\rangle\\
& = - c \,\|\theta-\theta^*\|^2 + \langle \theta-\theta^*, - c \,( \theta^* - \mathrm{pr}_D(\theta)) + f(\mathrm{pr}_D(\theta)) \rangle\\
& \le  - c \,\|\theta-\theta^*\|^2 +  c \,\|\theta-\theta^*\|\,\|\mathrm{pr}_D(\theta)-\theta^*  + \tfrac{1}{c}f(\mathrm{pr}_D(\theta))\|\\
& \le  - c \,\|\theta-\theta^*\|^2 +  c\, \sqrt{r_c}\,\|\theta-\theta^*\|\, \|\mathrm{pr}_D(\theta)-\theta^*\|\\
& \le -c\,(1-\sqrt r_c)\, \|\theta-\theta^*\|^2,
\end{align*}
which shows that $ f_c$ satisfies B.1(i) on $\R^d$ with $ L_c$ in place of $L$.  

Using the latter estimate, \eqref{aaa0} with $c_2'=1/L'$ and the Lipschitz continuity  of $\mathrm{pr}_D$  we get for any $\theta\in\R^d$ that
\begin{align*}
\|f_{c}(\theta)\|^2 + \tfrac{1}{ L_c'} \langle \theta-\theta^*,f_{c}(\theta)\rangle & \le  \|-c\, (\theta- \mathrm{pr}_D(\theta)) + f(\mathrm{pr}_D(\theta))\|^2 - \tfrac{ L_c}{ L_c'} \|\theta-\theta^*\|^2\\
& \le 2 c^2 \|\theta- \mathrm{pr}_D(\theta)\|^2 + 2 \|f(\mathrm{pr}_D(\theta))\|^2- \tfrac{ L_c}{\bar L_c'} \|\theta-\theta^*\|^2\\
& \le 2  c^2 \|\theta- \mathrm{pr}_D(\theta)\|^2 + \tfrac{2}{(L')^2}\|\mathrm{pr}_D(\theta)-\theta^*\|^2- \tfrac{ L_c}{ L_c'} \|\theta-\theta^*\|^2\\
& \le (2 c^2 + \tfrac{2}{(L')^2} -\tfrac{ L_c}{ L_c'} ) \|\theta- \theta^*\|^2 = 0,
\end{align*}
which shows that $ f_c$ satisfies B.1(ii) on $\R^d$ with $ L_c'$ in place of $L'$.

Finally, let $\theta\in \R^d\setminus D$, which implies that $\|\theta-\theta^*\|\ge \delta$. Using the latter fact and the projection property and the contractivity of $\mathrm{pr}_D$ we get
\begin{align*}
\| f_c(\theta) - H(\theta-\theta^*)\| & = \|-c(\theta-\mathrm{pr}_D(\theta)) +f(\mathrm{pr}_D(\theta)) - H(\mathrm{pr}_D(\theta)-\theta^*) - H(\theta-\mathrm{pr}_D(\theta))\|\\
& \le \| (c\,I_d+H)(\theta-\mathrm{pr}_D(\theta))\| + \|f(\mathrm{pr}_D(\theta)) - H(\mathrm{pr}_D(\theta)-\theta^*)\|\\
& \le \|c\,I_d+H\|\|\theta-\mathrm{pr}_D(\theta)\| + L''\,\|\mathrm{pr}_D(\theta)-\theta^*\|^{1+\lambda}\\& \le (c+\|H\|)\|\theta-\theta^*\|+ L''\,\|\theta-\theta^*\|^{1+\lambda}\\&
\le \bigl((c+\|H\|)\tfrac{1}{\delta^\lambda} + L''\bigr) \|\theta-\theta^*\|^{1+\lambda}.
\end{align*}
Observing that $\|H\|\le 1/L'$, see~\eqref{aaa3}, completes the proof of the lemma.
\end{proof}

Replacing $f$ by $ f_c$ in~\eqref{dynsys2} we obtain the dynamical system 
\begin{equation}\label{moddynsys}
\theta_{c,n}=\theta_{c,n-1}+\gamma_n \bigl(  f_c(\theta_{c,n-1})+ \eps_n\, R_n +\sig_n\, D_{n}\bigr),
\end{equation}
for $n\in\N$, where $\theta_{c,0}\in D$ is a deterministic starting value in $D$.

Employing Lemma~\ref{l234} we immediately arrive at the following fact.

\begin{ext}\label{Ext2}
Assume that  $B(\theta^*,\delta)\subset D$ for some $\delta\in (0,\infty)$. Then  Corollary~\ref{Cor3} remains valid for the modified Polyak-Ruppert algorithm 
\begin{equation}\label{modpolrup}
\bar \theta_{c,n} = \frac{1}{\bar b_n} \sum_{k=1}^{n}b_k \theta_{c,k},\qquad n\in\N,
\end{equation}
in place of the scheme~\eqref{average}, if $\R^d$ is replaced by $D$  in Assumption B.1   and $ c\in (1/(2L'),\infty)$ with $L'$ according to B.1(ii). 

Moreover, Theorem~\ref{thm2} remains valid for the scheme \eqref{modpolrup} as well if, additionally, Assumption B.3 is satisfied with $ L_c$ given by~\eqref{newpar} in place of $L$.
\end{ext}

Similar to Extension~\ref{Ext2} we can  extend Theorem~\ref{thm_Gilesnew_2}  on the multilevel Polyak-Ruppert averaging. To this end we define for $c\in(0,\infty)$ extensions $F_c,F_{c,1},F_{c,2},\dots\colon \R^d\times\mathcal U \to \R^d$ of the mappings $F,F_1,F_2\dots \colon D\times \mathcal U\to \R^d$  by taking
\[
G_c\colon \R^d\times \mathcal U\to \R^d,\quad (\theta,u)\mapsto -c(\theta-\pr_D(\theta)) + G(\pr_D(\theta),u)
\] 
for $G\in\{F,F_1,F_2,\dots\}$. Note that $\E[F_c(\theta,U)] = f_c(\theta)$ and $f(\theta^*)=f_c(\theta^*)=0$ with $f_c$ given by~\eqref{continue}. 

Clearly, if the mappings $F,F_1,F_2,\dots$ satisfy  C.1(i),(ii) on $D$ then the mappings $F_c,F_{c,1},F_{c,2},\dots$ satisfy C.1(i),(ii) on $\R^d$ with $\Gamma_1\circ \pr_D,\Gamma_2\circ \pr_D$ in place of $\Gamma_1,\Gamma_2$.
Furthermore, if $\Gamma_1,\Gamma_2$ satisfy Assumption C.2 on $D$ then $\Gamma_1\circ \pr_D,\Gamma_2\circ \pr_D$ satisfy Assumption C.2 on $\R^d$, since we have $\|\pr_D(\theta)\|\le  \|\theta\| +\|\pr_D(0)\|$ for every $\theta\in\R^d$. 

We thus take
\begin{align*}
Z_{c,n}(\theta) & =\sum_{k=1}^{m_n(\pr_D(\theta))} \frac{1}{N_{n,k}(\pr_D(\theta))} \sum_{\ell=1}^{N_{n,k}(\pr_D(\theta))} \bigl(F_{c,k}(\theta, U_{n,k,\ell})- F_{c,k-1}(\theta,U_{n,k,\ell})\bigr)\\
 & =  -c(\theta -\pr_D(\theta)) + Z_n(\pr_D(\theta)), 
\end{align*}
with $m_n$, $N_{n,k}$ and $Z_n$  given by~\eqref{maxlevel},\eqref{itnum} and ~\eqref{mul1},  respectively, as a multilevel approximation of $f_c(\theta)$ in
the $n$-th Robbins-Monro step, and we use the multilevel scheme 
\begin{equation}\label{rm10x}
\theta_{c,n} =\theta_{c,n-1}+\gamma_n\, Z_{c,n}(\theta_{c,n-1})
\end{equation}
for Polyak-Ruppert averaging.

Employing Lemma~\ref{l234} we  get the following result.

\begin{ext}\label{Ext4}
Assume that  $B(\theta^*,\delta)\subset D$ for some  $\delta\in (0,\infty)$. Then  Theorem~\ref{thm_Gilesnew_2} remains valid for the modified Polyak-Ruppert algorithm 
\begin{equation}\label{modpolrup2}
\bar \theta_{c,n} = \frac{1}{\bar b_n} \sum_{k=1}^{n}b_k \theta_{c,k},\qquad n\in\N,
\end{equation}
with $(\theta_{c,n})_{n\in\N}$ given by~\eqref{rm10x} in place of the scheme~\eqref{average},
if $\R^d$ is replaced by $D$ in Assumptions B.1, C.1, C.2 and $ c\in (1/(2L'),\infty)$ with $L'$ according to B.1(ii). 
\end{ext}

\section{Numerical Experiments}\label{sec5}
We illustrate the application of our multilevel methods in  the simple case of computing the volatility in a Black Scholes model based on the price of a European call.

Fix $T,\mu,s_0,K\in (0,\infty)$ and let $W$ denote a one-dimensional Brownian motion on $[0,T]$. For every $\theta\in (0,\infty)$ we use $S^\theta$ to denote the geometric Brownian motion on $[0,T]$ with initial value $s_0$, trend $\mu$ and volatility $\theta$, i.e.
\begin{equation}\label{bs}
\begin{aligned}
S^\theta_0 & = s_0,\\
dS^\theta_t & = \mu S^\theta_t\, dt + \theta S^\theta_t\, dW_t,\quad t\in[0,T].
\end{aligned}
\end{equation}
In a Black Scholes model with fixed interest rate $\mu$ the fair price of a  European call with maturity $T$, strike $K$ and underlying geometric Brownian motion with volatility $\theta$ is  given by 
\[
p(\theta) = \E[C(\theta,W)], 
\]
where 
\[
C(\theta,W) = \exp(-\mu T)(S_T^\theta - K)_+,
\]
and  according to the Black-Scholes formula $p$  satisfies
\[
p(\theta) =  s_0\,\Phi\Bigl(\tfrac{\ln (s_0/K) + (\mu+\theta^2/2)T}{\theta\sqrt {T}}\Bigr) - \exp(-\mu T)K \Phi\Bigl(\tfrac{\ln (s_0/K) + (\mu-\theta^2/2)T}{\theta\sqrt {T}}\Bigr),
\]
where $\Phi$ denotes the standard normal distribution function.

Fix $\vartheta_1 < \vartheta_2$ as well as $\theta^*\in [\vartheta_1,\vartheta_2]$. Our computational goal is to approximate $\theta^*$ based on the knowledge of $\vartheta_1,\vartheta_2$ and the value of the price $p(\theta^*)$. 

Within the framework of sections~\ref{sec3} and~\ref{sec4} we  take $d=1$, $D=[\vartheta_0,\vartheta_1]$, $U=W$ and
\[
F(\theta,W) = p(\theta^*) - C(\theta,W),\quad \theta\in D.
\] 
Moreover, we approximate $F(\theta,W)$ by employing equidistant Milstein schemes: for $M,k\in\N$ with $M\ge 2$ and $\theta\in D$ we define
\[
F_{M,k}(\theta,W) = p(\theta^*) - \exp(-\mu T)(\widehat S^\theta_{M^k,T}-K)_+,
\]
where $\widehat S^\theta_{M^k,T}$ denotes the Milstein approximation of $S^\theta_T$ based on $M^k$ equidistant steps, i.e.
\[
\widehat S^\theta_{M^k,T} = s_0\prod_{\ell=1}^{M^k} \Bigl(1+\mu \tfrac{ T}{M^k} + \theta \Delta_\ell W + \tfrac{\theta^2}{2}\bigl((\Delta_\ell W)^2- \tfrac{T}{M^k}\bigr)\Bigr) 
 \] 
with $\Delta_\ell W = W(\ell T/M^k) - W((\ell-1)T/M^k)$.

We briefly check the validity of Assumptions B.1, C.1 and C.2. Clearly, the mapping $f = \E[F(\cdot,W)]\colon D\to \R$ satisfies
\[
f(\theta) = p(\theta^*) - p(\theta),\quad \theta\in D.
\]
Note that $p$ is two times differentiable with respect to $\theta$ on $(0,\infty)$ with
\begin{equation}\label{num1}
\begin{aligned}
\tfrac{\partial p}{\partial \theta}(\theta) & = s_0 \sqrt{T} \varphi\Bigl(\tfrac{\ln (s_0/K) + (\mu+\theta^2/2)T}{\theta\sqrt {T}}\Bigr),\\
\tfrac{\partial^2 p}{\partial \theta^2}(\theta) & = \tfrac{(\ln (s_0/K) + (\mu+\theta^2/2)T)(\ln (s_0/K) + (\mu-\theta^2/2)T)}{\theta^3 T} \tfrac{\partial p}{\partial \theta}(\theta),
\end{aligned}
\end{equation}
where $\varphi$ denotes the density of  the standard normal distribution.
Let $g(\theta) = \frac{\ln(s_0/K) + (\mu +\theta^2/2)T}{\theta\sqrt{T}}$ and put $u= 2(\ln(s_0/K)+ \mu T)/T$ as well as
\[
z^* =  s_0\sqrt{T} \begin{cases} \min\bigl(\varphi(g(\vartheta_1)),\varphi(g(\vartheta_2))\bigr), & \text{if }u \not\in (\vartheta_1^2,\vartheta_2^2),\\
\min\bigl(\varphi(\max(g(\vartheta_1),g(\vartheta_2))),\varphi(\sqrt{u})\bigr) & \text{if }u \in (\vartheta_1^2,\vartheta_2^2).
\end{cases}
\]
Using~\eqref{num1} it is then straightforward to verify 
that $f$ satisfies Assumption B.1 on $D$ with parameters 
\begin{equation}\label{num2}
L= s_0 \sqrt{T}\min_{\theta\in[\vartheta_1,\vartheta_2]}\varphi( g(\theta)) =  z^*
\end{equation}
and
\begin{equation}\label{num2a}
 L' = \frac{\sqrt{2\pi}}{s_0\sqrt{T}},\quad H=-\frac{\partial p}{\partial \theta}(\theta^*),\quad L'' = \max_{\theta\in[\vartheta_1,\vartheta_2]}\Bigl|\frac{\partial^2 p}{\partial \theta^2}(\theta)\Bigr|,\quad \lambda = 1.
\end{equation}
 As is well known there exists a constant $c(T,\mu,\vartheta_1,\vartheta_2)\in (0,\infty)$, which depends only on $T,\mu,\vartheta_1,\vartheta_2$, such that
 \[
\sup_{\theta\in D} \E[|\widehat S^\theta_{M^k,T}- S^\theta_T|^2]^{1/2} \le c(T,\mu,\vartheta_1,\vartheta_2) M^{-k}
 \]
Since $|F_{M,k}(\theta,W)-F(\theta,W)|\le |\widehat S^\theta_{M^k,T}- S^\theta_T|$ we conclude that Assumption C.1 is satisfied on $D$ with parameters
 \[
 \alpha = \beta = 1,\quad \Gamma_1 = \Gamma_2 = c(T,\mu,\vartheta_1,\vartheta_2,M)
 \]
for some constant $c(T,\mu,\vartheta_1,\vartheta_2,M)\in (1,\infty)$, which depends only on $T,\mu,\vartheta_1,\vartheta_2,M$. Consequently, Assumption C.2 is satisfied on $D$ as well.
 
First, we consider the projected multilevel Robbins-Monro scheme \eqref{rmproj} with
step-size $\gamma_n$, noise-level $\sigma_n$ and bias-level $\eps_n$  given by
\[
\gamma_n = \tfrac{2}{L n},\quad \sigma_n^2 = \tfrac{\sqrt{c(T,\mu,\vartheta_1,\vartheta_2,M)}}{n^3},\quad \eps_n = \tfrac{c(T,\mu,\vartheta_1,\vartheta_2,M)}{n^2}.
\]  
Note that the constant $c(T,\mu,\vartheta_1,\vartheta_2,M)$ does not need to be known in order to implement the scheme. We have  
\[
\theta_n = \proj_{[\vartheta_1,\vartheta_2]}\bigl(\theta_{n-1} + \tfrac{2}{L n} Z_n(\theta_{n-1})\bigr),
\]
where $\theta_0 \in [\vartheta_1,\vartheta_2]$ and for all $\theta\in D$
\begin{equation}\label{zz}
Z_n(\theta) = \sum_{k=1}^{1\vee \lceil 2 \log_M (n)\rceil}\frac{1}{\lceil n^3 M^{-3k/2}\rceil}
\sum_{\ell=1}^{\lceil n^3 M^{-3k/2}\rceil} \bigl(F_{M,k}(\theta,W_{n,k,\ell}) - F_{M,k-1}(\theta,W_{n,k,\ell})\bigr)
\end{equation}
with independent copies $W_{n,k,\ell}$ of $W$. Then by Extension~\ref{Ext3} of Theorem~\ref{thm_Gilesnew} there exists $\kappa\in (0,\infty)$ such that for every $n\in\N$,
\begin{equation}\label{theo}
\E [\|\theta_n-\theta^*|^2]^{1/2} \le \kappa\, n^{-2},\quad \cost_n \le \kappa\, n^4.
\end{equation}

In the following we use the model parameters 
\begin{equation}\label{setting}
s_0 = 10, T = 2,\, \mu= 0.01 ,\, K = 11 ,\, \vartheta_1 = 0.05 ,\,  \vartheta_2 = 0.5,\, \theta^* = 0.2
\end{equation}
and we choose 
\begin{equation}\label{setting2}
M = 4,\, \theta_0 = 0.1
\end{equation}
in the definition of the Robbins-Monro scheme. 

Figure~\ref{fig1} shows  a log-log plot of  a simulation  of the first $500$ steps of the error process $(|\theta_n - \theta^*|)_{n\in\N_0}$. 
\begin{figure}[!h]
\centering
\includegraphics[scale = .6]{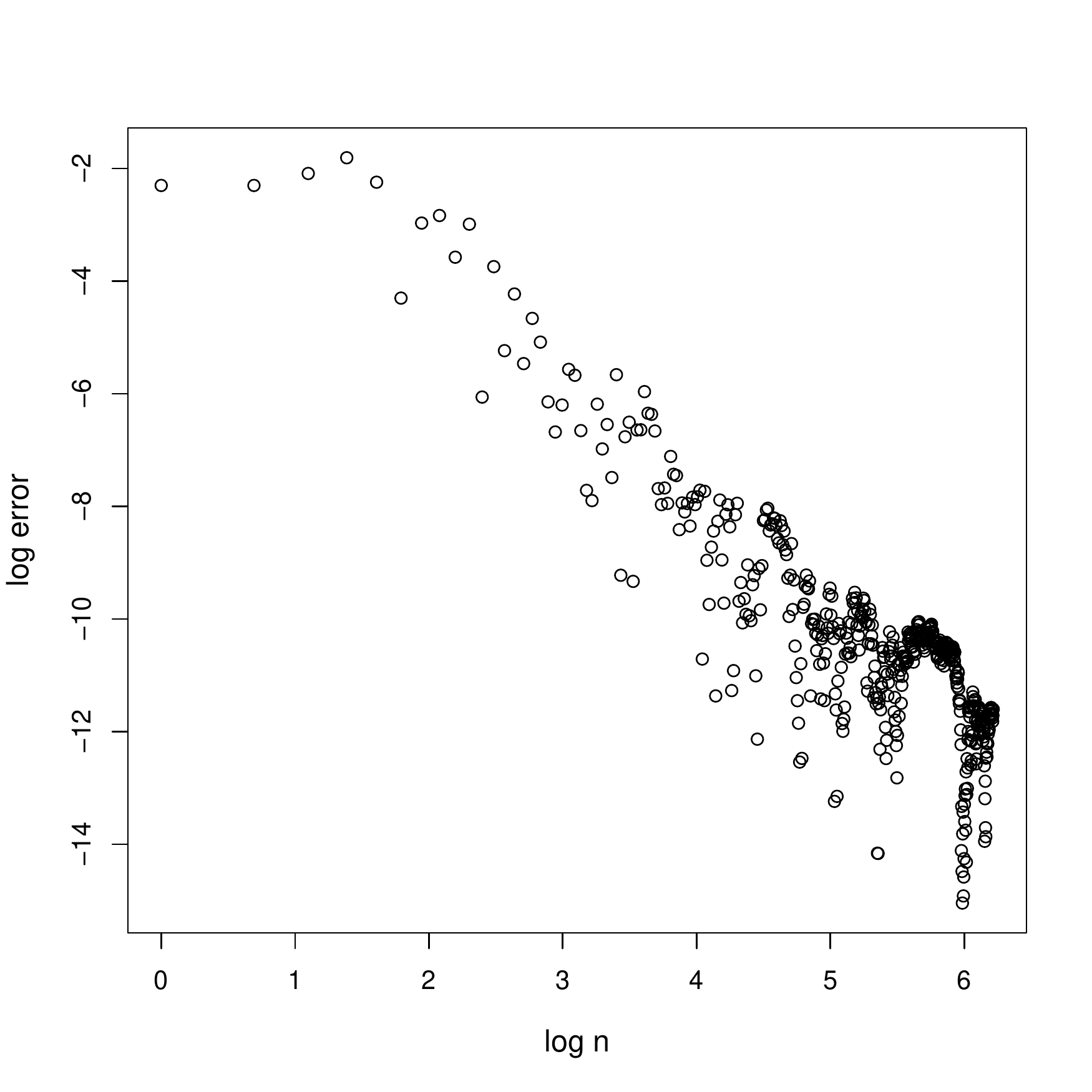}
\caption{Multilevel Robbins-Monro: error trajectory for $n=1,\dots,500$}\label{fig1}
\end{figure}

Figure~\ref{fig1a} 
shows the log-log plot of Monte Carlo estimates of the root mean squared error of $\theta_{n}$ and the corresponding average computational times for $n=1,\dots,100$ based on $N= 200 $ replications. 
Both plots are in accordance with the theoretical bounds in~\eqref{theo}. 
\begin{figure}[!h]
\centering
\includegraphics[scale = .45]{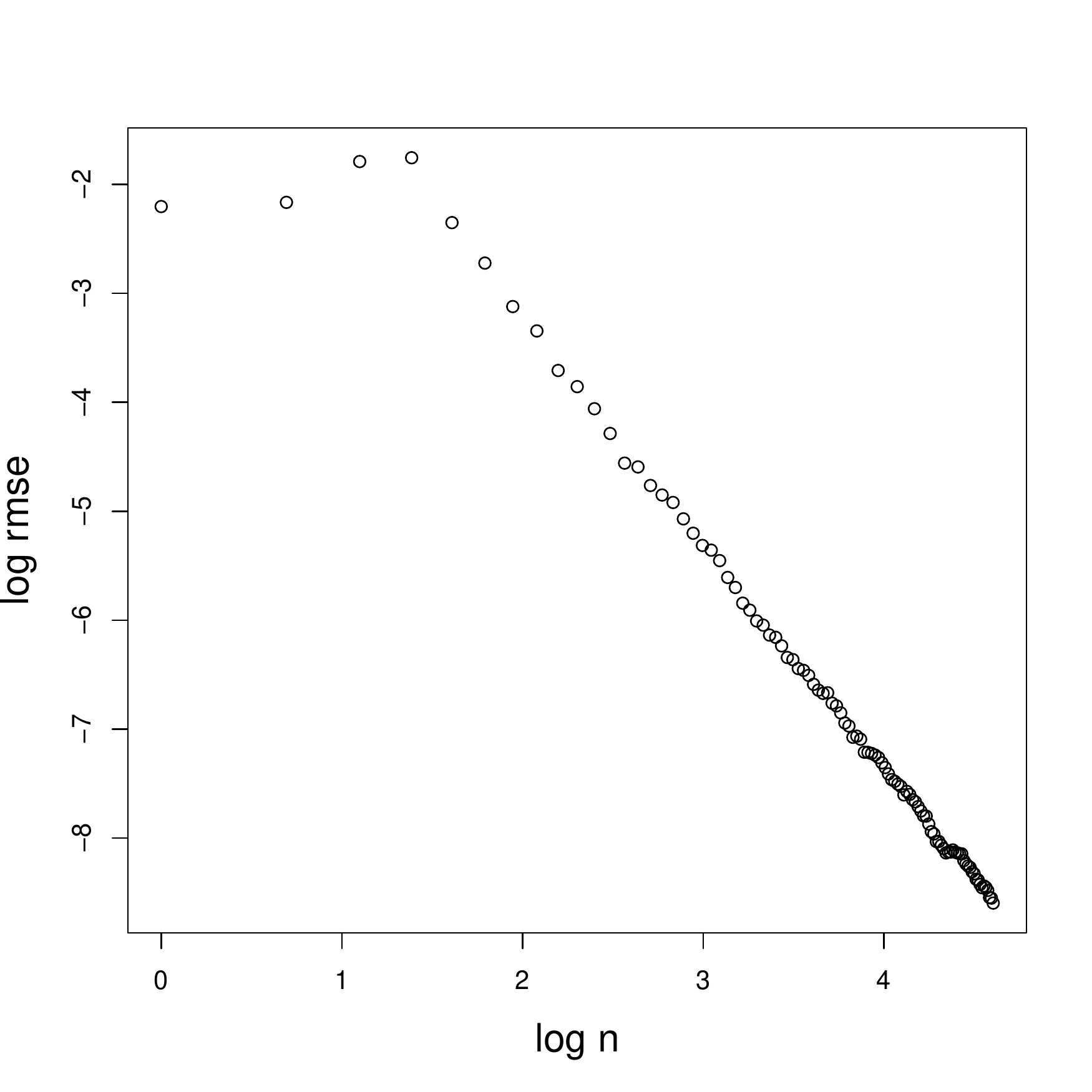}\hspace*{-.1cm}
\includegraphics[scale = .45]{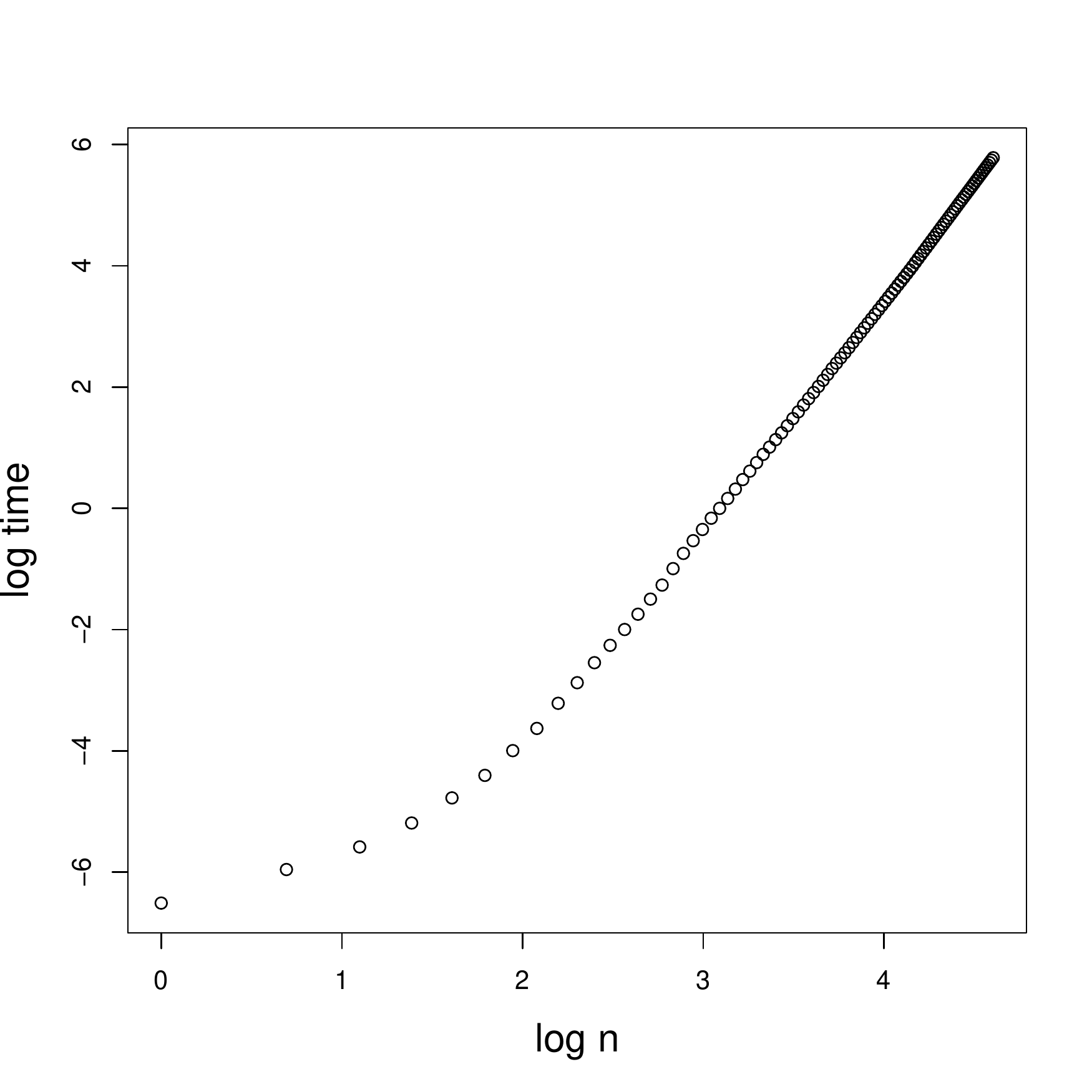}
\caption{Multilevel Robbins-Monro: estimated root mean squared error and average computational time for $n=1,\dots,500$.}\label{fig1a}
\end{figure}

Next, we consider the multilevel Polyak-Rupert averaging~\eqref{modpolrup2} with step-size $\gamma_n$, noise-level $\sigma_n$, bias-level $\eps_n$, weight $b_n$ and extension parameter $c$  given by
\[
\gamma_n = \tfrac{1}{n^{ 0.9}},\quad \sigma_n^2 = \tfrac{\sqrt{c(T,\mu,\vartheta_1,\vartheta_2,M)}}{n^3},\quad \eps_n = \tfrac{c(T,\mu,\vartheta_1,\vartheta_2,M)}{n^2},\quad b_n = n^2,\quad c=\frac{1}{L'}.
\] 
Thus 
\[
\bar \theta_{c,n} = \frac{6}{n(n+1)(2n+1)}\sum_{k=1}^n k^2\, \theta_{k,c},
\]
where
\[
\theta_{n,c} = \theta_{n-1,c} + \frac{1}{n^{ 0.9} } \Bigl(-\frac{1}{L'}(\theta_{c,n-1}- \proj_{[\vartheta_1,\vartheta_2]}(\theta_{c,n-1})) + Z_n(\proj_{[\vartheta_1,\vartheta_2]}(\theta_{c,n-1}))\Bigr), 
\]
with $Z_n$ given by~\eqref{zz} and a deterministic $\theta_{0,c} \in [\vartheta_1,\vartheta_2]$. 
Then by Extension~\ref{Ext4} of Theorem~\ref{thm_Gilesnew_2} there exists for every $q\in [1,2)$ a constant $\kappa\in (0,\infty)$ such that for every $n\in\N$,
\begin{equation}\label{theo2}
\E [\|\bar \theta_{c,n}-\theta^*|^q]^{1/q} \le \kappa n^{-2},\quad \cost_n \le \kappa n^4.
\end{equation} 

We choose the parameters $s_0, T,\mu, K, \vartheta_1,\vartheta_2, M, \theta_{c,0} = \theta_0$ as in~\eqref{setting} and~\eqref{setting2}. Figure~\ref{fig2} shows the log-log plot of a trajectory of the  error process $(|\bar \theta_{c,n} - \theta^*|)_{n\in\N_0}$ until $n=500$. 
\begin{figure}[!h]
\centering
\includegraphics[scale = .6]{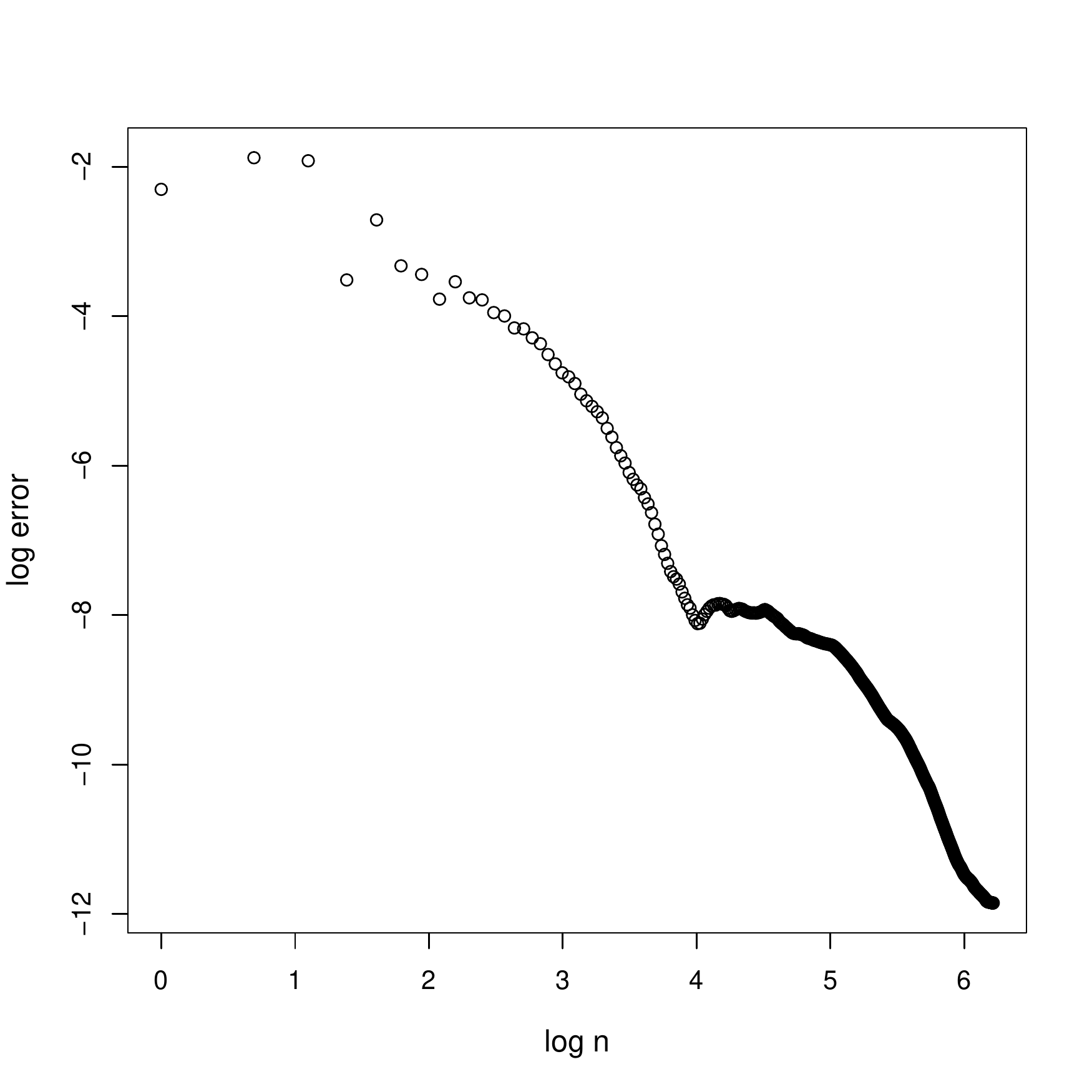}%\vspace*{-.5cm}
\caption{Multilevel Polyak-Ruppert: error trajectory for $n=1,\dots,500$}\label{fig2}
\end{figure}
Figure~\ref{fig2a} 
shows the log-log plot of Monte Carlo estimates of the root mean squared error of $\bar \theta_{c,n}$ and the corresponding average computational times for $n=1,\dots,100$ based on $N= 200 $ replications. \black
As for the multilevel Robbins-Monro scheme, both plots are in accordance with the theoretical bounds in~\eqref{theo2}. 
\begin{figure}[!h]
\centering
%\hspace*{-.5cm}
\includegraphics[scale = .45]{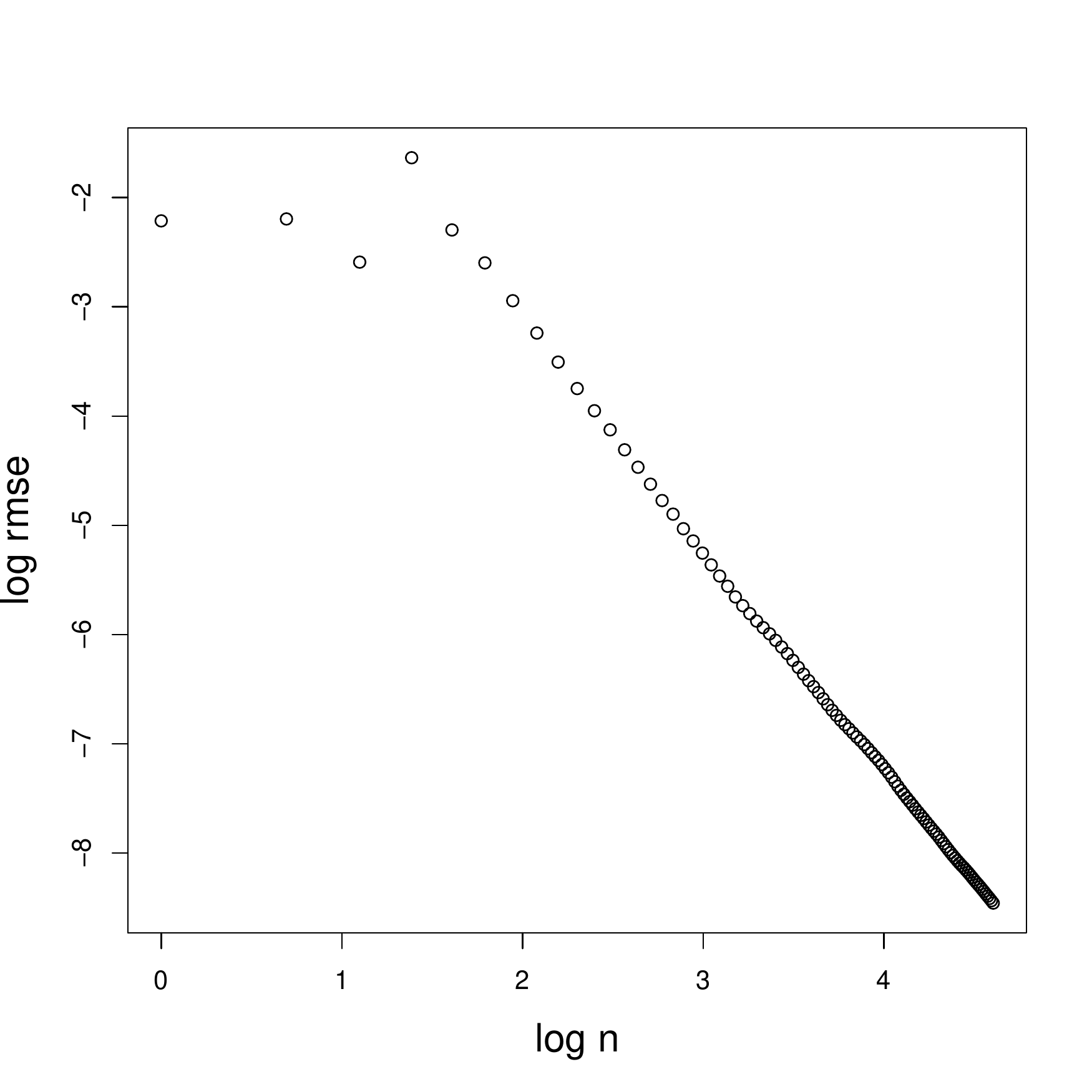}\hspace*{-.1cm}
%\caption{Multilevel Robbins-Monro: estimated root mean squared error for $n=1,\dots,500$}%\label{fig1a}
\includegraphics[scale = .45]{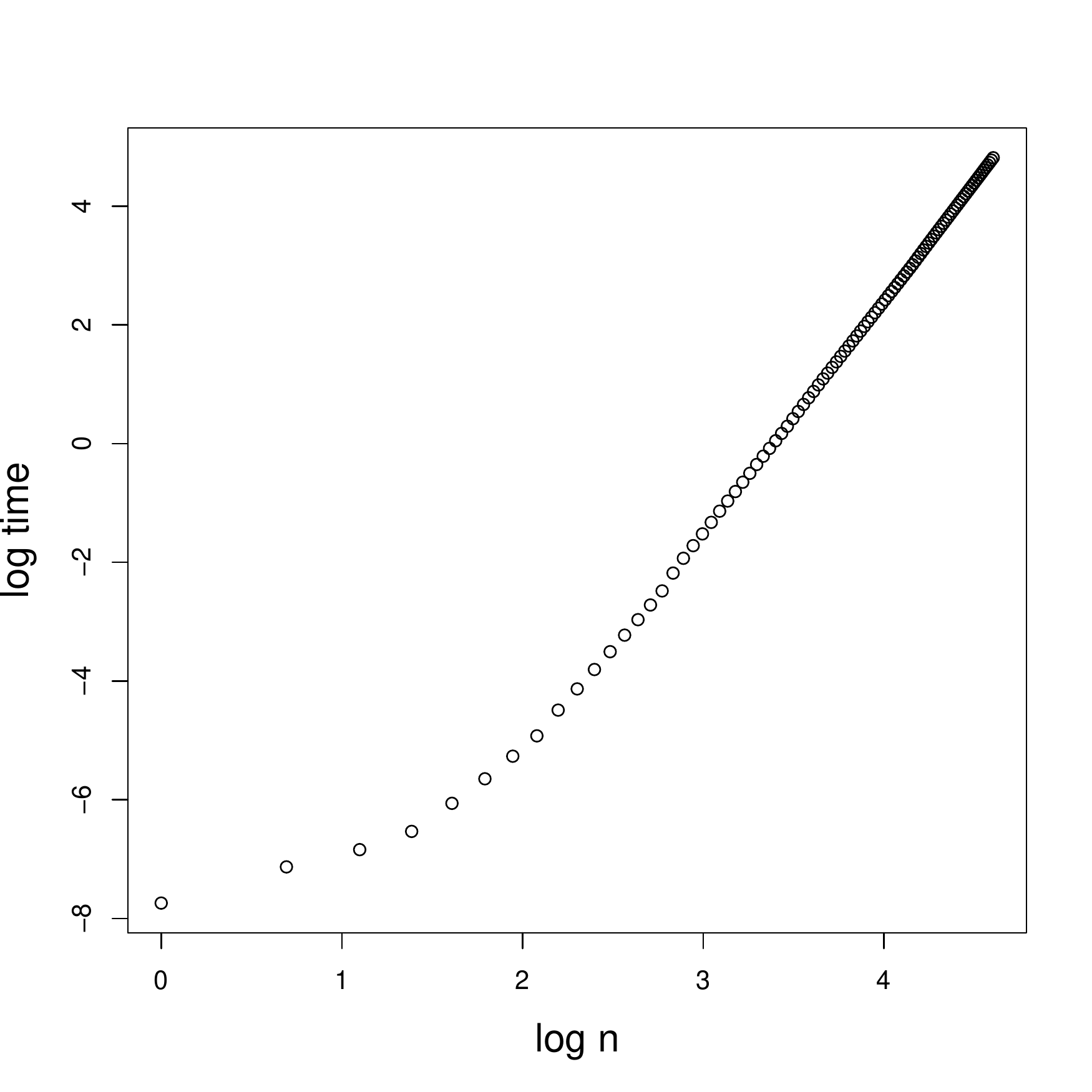}
\caption{Multilevel Polyak Ruppert: estimated root mean squared error and average computational time for $n=1,\dots,500$.}\label{fig2a}
\end{figure}

\black

\newpage

\section*{Appendix}

Let $(\Omega,\cF,\P)$ be a probability space endowed with a filtration $(\cF_n)_{n\in\N_0}$ and let $\|\cdot\|$ denote a Hilbert space norm on $\R^d$. \smallskip

In this section we provide $p$-th mean estimates for an adapted $d$-dimensional dynamical system $(\zeta_n)_{n\in\N_0}$  with the property that for each $n\in\N$, $\zeta_n$ is a zero-mean perturbation of a previsible proposal $\xi_n$ being comparable in size to $\zeta_{n-1}$. More formally, we assume that there exist a previsible $d$-dimensional process $(\xi_n)_{n\in\N}$, a $d$-dimensional martingale $(M_n)_{n\in\N_0}$ with $M_0=\zeta_0$ and a constant~$c\geq 0$ such that for all $n\in\N$
\begin{equation}\label{trip1} 
\begin{aligned}
\zeta_{n} & = \xi_{n}+\Delta M_n,\\
\|\xi_n\| &\leq \|\zeta_{n-1}\|\vee c,
\end{aligned}
\end{equation}
where $\Delta M_n = M_n-M_{n-1}$. Note that necessarily $\xi_n=\E[\zeta_n|\cF_{n-1}]$. 

\begin{theorem}\label{thm_BDG}
Assume that $(\zeta_n)_{n\in\N_0}$ is an adapted $d$-dimensional process, which satisfies ~(\ref{trip1}), and let $p\in [1,\infty)$. Then there exists a constant $\kappa\in (0,\infty)$, which only depends on  $p$, such that for every $n\in\N_0$, 
\[
\E\bigl[\max_{0\le k\le n}\|\zeta_k\|^p\bigr]\leq \kappa\, \bigl( \E\bigl[ [M]_n^{p/2}\bigr] +  c^p\bigr),
\]
where
\[
[M]_n=\sum_{k=1}^{n} \|\Delta M_k\|^2 +\|M_0\|^2.
\]
\end{theorem}

\begin{proof} Fix $p\in [1,\infty)$. 

We first consider the case where $c=0$.  
Recall that by the BDG inequality there exists a constant $\bar \kappa>0$ depending only on  $p$ such that for every $d$-dimensional martingale $(M_n)_{n\in\N_0}$,
\[
\E\bigl[\max_{0\le k\le n}\|M_k\|^p\bigr]\leq \bar \kappa\, \E\bigl[ [M]_n^{p/2}\bigr].
\]
We fix a time horizon $T\in\N_0$ and prove the statement of the theorem with $\kappa = \bar \kappa$ by induction: we say that the statement holds up to time $t\in\{0,\dots,T\}$, if for every  $d$-dimensional adapted process $(\zeta_n)_{n\in\N_0}$, for every $d$-dimensional previsible process   $(\xi_n)_{n\in\N}$ and for every  $d$-dimensional martingale $(M_n)_{n\in\N_0}$  with  
\begin{align}
\begin{cases}
\zeta_0=M_0,\\
\|\xi_{n}\|  \le \|\zeta_{n-1}\|, &\text{ if } 1\le n\leq t,\\
\zeta_{n}  = \xi_{n}+\Delta M_n, &\text{ if }1\le n\leq t,\\
\zeta_{n}=\zeta_{n-1}+ \Delta M_n, & \text{ if }n> t, \tag{$C_t$}
\end{cases}
\end{align}
one has
\[
\E\bigl[\max_{0\le n\le T}\|\zeta_n\|^p\bigr]\leq \bar \kappa\, \E\bigl[ [M]_T^{p/2}\bigr].
\]
Clearly, the statement is satisfied up to time $0$ as a consequence of the BDG inequality. Next, suppose that the statement is satisfied up to time $t\in\{0,\dots,T-1\}$. Let  $(\zeta_n)_{n\in\N_0}$ be a $d$-dimensional adapted  process, $(\xi_n)_{n\in\N}$ be a  $d$-dimensional previsible process and  $(M_n)_{n\in\N_0}$ be a $d$-dimensional martingale satisfying property ($C_{t+1}$).
Consider any $\cF_{t}$-measurable random orthonormal transformation $U$ on $(\R^d,\|\cdot\|)$ and put 
\[
\zeta^U_n=\begin{cases} \zeta_n,&\text{ if }n\leq t,\\
\zeta_{t}+U(M_{n}-M_{t}), &\text{ if }n>t\end{cases}
\]
as well as
\[
M^U_n=\begin{cases} M_n,&\text{ if }n\leq t,\\
M_{t}+ U(M_{n}-M_{t}), &\text{ if }n>t.\end{cases}
\]
Then it is easy to check that $(M^U_n)_{n\in\N_0}$ is a martingale with $[M^U]_n = [M]_n$ for all $n\in\N$. Furthermore, $(\zeta^U_n)_{n\in\N_0}$ is adapted  and the triple $(\zeta^U,\xi, M^U)$ satisfies property ($C_t$). Hence, by the induction hypothesis,
\begin{equation}\label{a1}
\E\bigl[\max_{0\le n\le T}\|\zeta^U_n\|^p\bigr] \leq \bar \kappa\, \E\bigl[ [M^U]_T^{p/2}\bigr] =\bar  \kappa\, \E\bigl[ [M]_T^{p/2}\bigr].
\end{equation}
Note that for any such random orthonormal transformation $U$, the norm of the random variable $\zeta_n^U$ is the same as the norm of the variable $\bar\zeta_n^U$ given by
\[
\bar \zeta^U_n=\begin{cases} \zeta_n,&\text{ if }n\leq t,\\
U^* \zeta_{t}+ M_{n}-M_{t}, &\text{ if }n>t,\end{cases}
\]
whence
\begin{equation}\label{a2}
\E\bigl[\max_{0\le n\le T}\|\bar\zeta^U_n\|^p\bigr] = \E\bigl[\max_{0\le n\le T}\|\zeta^U_n\|^p\bigr].
\end{equation}
Clearly, we can choose an $\cF_{t}$-measurable random orthonormal transformation $U$ on $(\R^d,\|\cdot\|)$ such that 
\[
U^*  \zeta_t = \frac{\|\zeta_t\|}{\|\xi_{t+1}\|} \xi_{t+1} 
\]
holds on $\{\xi_{t+1}\neq 0\}$. Let
\[
\alpha = \frac{\|\xi_{t+1}\|+\|\zeta_t\|}{2\|\zeta_t\|}\cdot 1_{\{\zeta_t\neq 0\}}.
\]
Then $\alpha$ is $\cF_{t}$-measurable and takes values in $[0,1]$ since $\|\xi_{t+1}\|\le \|\zeta_t\|$. Moreover, we have $\xi_{t+1} = \alpha U^* \zeta_t + (1-\alpha)  (-U)^* \zeta_t$ so that by property ($C_{t+1}$) of the triple $(\zeta,\xi,M)$,
\[
\zeta_n= \xi_{t+1} + M_n-M_t = \alpha \bar \zeta^U_n+(1-\alpha) \bar\zeta_n^{-U}
\]
for $n= t+1,\dots,T$. Note that $\zeta_n=\zeta^U_n=\zeta_n^{-U}$ for $n=0,\dots,t$. By  convexity of $\|\cdot\|^p$ we thus obtain
\begin{align*}
\max_{0\le n\le T}\|\bar\zeta^U_n\|^p & = \max_{0\le n\le T}\|\alpha \bar\zeta^U_n+  (1-\alpha) \bar\zeta^{-U}_n\|^p
 \le \alpha \max_{0\le n\le T}\|\bar\zeta^U_n\|^p +  (1-\alpha)\max_{0\le n\le T}\| \bar\zeta^{-U}_n\|^p.
\end{align*}
Hence
\begin{align*}
\E\bigl[\max_{0\le n\le T}\|\zeta_{n}\|^p|\cF_{t}\bigr] & \le \alpha \E\bigl[\max_{0\le n\le T}\|\bar\zeta^U_n\|^p|\cF_{t}\bigr] +  (1-\alpha)\E\bigl[\max_{0\le n\le T}\| \bar\zeta^{-U}_n\|^p|\cF_{t}\bigr]  \le \E\bigl[\max_{0\le n\le T}\|\bar\zeta^{U'}_n\|^p|\cF_{t}\bigr],
\end{align*}
where $U'$ is the $\cF_{t}$-measurable random orthonormal transformation given by
\[
U'(\omega)= \begin{cases} U(\omega)& \text{ if }\omega\in\bigl\{ \E[\max_{0\le n\le T}\|\bar\zeta^U_n\|^p|\cF_{t}]\geq \E[\max_{0\le n\le T}\|\bar\zeta^{- U}_n\|^p|\cF_{t}]\bigr\},\\
-U(\omega) & \text{ otherwise}.
\end{cases}
\] 
Applying \eqref{a1} and \eqref{a2} with $U=U'$
finishes the induction step.\smallskip

Next, we consider the case of $c > 0$. Suppose that $\zeta,\xi$ and $M$  are as stated in the theorem. For $n\in\N$ we put
\[
\tilde \xi_n = (1-c/\|\xi_n\|)_+ \cdot \xi_n
\]
and 
\[
\tilde \zeta_n = \tilde \xi_n + \Delta  M_n.
\]
Furthermore, let $\tilde \zeta_0=\zeta_0=M_0$. We will show that the triple $(\tilde \zeta,\tilde \xi,M)$ satisfies~(\ref{trip1}) with $c=0$.
Clearly, $(\tilde\zeta_n)_{n\in\N_0}$ is adapted and $(\tilde \xi_n)_{n\in\N}$ is previsible. Moreover, one has  for $n\in\N$  on $\{\|\xi_n\|\geq c\}$ that
\begin{align*}
\|\tilde \xi_n\| & = \|\xi_n\|-c\le \|\zeta_{n-1}\|-c=\|\tilde \zeta_{n-1}+\xi_{n-1}-\tilde \xi_{n-1}\|-c\\
&\leq \|\tilde \zeta_{n-1}\|+\|\xi_{n-1}-\tilde \xi_{n-1}\|-c  =  \|\tilde \zeta_{n-1}\|
\end{align*}
and on $\{\|\xi_n\|< c\}$ that $\|\tilde \xi_n\|=0\leq \|\tilde \zeta_{n-1}\|$.
We may thus apply Theorem \ref{thm_BDG} with $c=0$ to obtain that  for every $n\in\N$,
\[
\E\bigl[\max_{0\le k \le n}\|\tilde \zeta_n\|^p\bigr]\leq \bar \kappa\,  \E\bigl[ [ M]_n^{p/2}\bigr].
\]
Since for every $n\in\N$,
\[
\|\zeta_n\|^p = \|\tilde \zeta_n + \xi_n-\tilde \xi_n\|^p \le 2^p(\|\tilde \zeta_n\|^p + c^p),
\]
we conclude  that
\[
\E\bigl[\max_{0\le k \le n}\|\zeta_n\|^p\bigr]   \le  2^p\bigl(\bar \kappa\,  \E\bigl[ [ M]_n^{p/2}\bigr] + c^p\bigr) \le
2^p(\bar \kappa \vee 1) \cdot \bigl( \E\bigl[ [ M]_n^{p/2}\bigr] + c^p\bigr),
\]
which completes the proof.
\end{proof}

%\section*{Acknowledgement}

\bibliographystyle{plain}
\bibliography{stoch_approx}

\begin{thebibliography}{10}

\bibitem{BMP90}
A.~Benveniste, M.~M{\'e}tivier, and P.~Priouret.
\newblock {\em Adaptive algorithms and stochastic approximations}, volume~22 of
  {\em Applications of Mathematics (New York)}.
\newblock Springer-Verlag, Berlin, 1990.

\bibitem{DiRe97}
J.~Dippon and J.~Renz.
\newblock Weighted means in stochastic approximation of minima.
\newblock {\em SIAM J. Control Optim.}, 35(5):1811--1827, 1997.

\bibitem{DMP08}
K.~Djeddour, A.~Mokkadem, and M.~Pelletier.
\newblock On the recursive estimation of the location and of the size of the
  mode of a probability density.
\newblock {\em Serdica Math. J.}, 34(3):651--688, 2008.

\bibitem{Duf96}
M.~Duflo.
\newblock {\em Algorithmes stochastiques}, volume~23 of {\em Math\'ematiques \&
  Applications (Berlin) [Mathematics \& Applications]}.
\newblock Springer-Verlag, Berlin, 1996.

\bibitem{Fri15}
N.~Frikha.
\newblock Multi-level stochastic approximation algorithms.
\newblock {\em Ann. Appl. Probab.}, 26:933--985, 2016.

\bibitem{GaKr74}
V.~F. Gaposhkin and T.~P. Krasulina.
\newblock On the law of the iterated logarithm in stochastic approximation
  processes.
\newblock {\em Theory Prob. Appl.}, 19(4):844--850, 1974.

\bibitem{Gil08}
M.~B. Giles.
\newblock Multilevel {M}onte {C}arlo path simulation.
\newblock {\em Oper. Res.}, 56(3):607--617, 2008.

\bibitem{Hei01}
S.~Heinrich.
\newblock Multilevel {M}onte {C}arlo methods.
\newblock In {\em Large-scale scientific computing}, pages 58--67. Springer
  Berlin Heidelberg, 2001.

\bibitem{KoTs04}
V.~R. Konda and J.~N. Tsitsiklis.
\newblock Convergence rate of linear two-time-scale stochastic approximation.
\newblock {\em Ann. Appl. Probab.}, 14(2):796--819, 2004.

\bibitem{KuYa93}
H.~J. Kushner and J.~Yang.
\newblock Stochastic approximation with averaging of the iterates: optimal
  asymptotic rate of convergence for general processes.
\newblock {\em SIAM J. Control Optim.}, 31(4):1045--1062, 1993.

\bibitem{KY03}
H.~J. Kushner and G.~G. Yin.
\newblock {\em Stochastic approximation and recursive algorithms and
  applications}, volume~35 of {\em Applications of Mathematics (New York)}.
\newblock Springer-Verlag, New York, second edition, 2003.
\newblock Stochastic Modelling and Applied Probability.

\bibitem{Lai03}
T.~L. Lai.
\newblock Stochastic approximation.
\newblock {\em Ann. Statist.}, 31(2):391--406, 2003.
\newblock Dedicated to the memory of Herbert E. Robbins.

\bibitem{LaiRob78}
T.~L. Lai and H.~Robbins.
\newblock Limit theorems for weighted sums and stochastic approximation
  processes.
\newblock {\em Proc. Nat. Acad. Sci. U.S.A.}, 75, 1978.

\bibitem{LBN94}
A.~Le~Breton and A.~Novikov.
\newblock Averaging for estimating covariances in stochastic approximation.
\newblock {\em Math. Methods Statist.}, 3(3):244--266, 1994.

\bibitem{LBN95}
A.~Le~Breton and A.~Novikov.
\newblock Some results about averaging in stochastic approximation.
\newblock {\em Metrika}, 42(3-4):153--171, 1995.
\newblock Second International Conference on Mathematical Statistics (Smolenice
  Castle, 1994).

\bibitem{LPW92}
L.~Ljung, G.~Pflug, and H.~Walk.
\newblock {\em Stochastic approximation and optimization of random systems},
  volume~17 of {\em DMV Seminar}.
\newblock Birkh\"auser Verlag, Basel, 1992.

\bibitem{MP11}
A.~Mokkadem and M.~Pelletier.
\newblock A generalization of the averaging procedure: the use of
  two-time-scale algorithms.
\newblock {\em SIAM J. Control Optim.}, 49(4):1523--1543, 2011.

\bibitem{Pel98b}
M.~Pelletier.
\newblock On the almost sure asymptotic behaviour of stochastic algorithms.
\newblock {\em Stochastic Process. Appl.}, 78(2):217--244, 1998.

\bibitem{Pel98a}
M.~Pelletier.
\newblock Weak convergence rates for stochastic approximation with application
  to multiple targets and simulated annealing.
\newblock {\em Ann. Appl. Probab.}, 8(1):10--44, 1998.

\bibitem{Pol90}
B.~T. Polyak.
\newblock A new method of stochastic approximation type.
\newblock {\em Avtomat. i Telemekh.}, 51(7):937--1008, 1998.

\bibitem{RM51}
H.~Robbins and S.~Monro.
\newblock A stochastic approximation method.
\newblock {\em Ann. Math. Statistics}, 22:400--407, 1951.

\bibitem{Rup82}
D.~Ruppert.
\newblock Almost sure approximations to the {R}obbins-{M}onro and
  {K}iefer-{W}olfowitz processes with dependent noise.
\newblock {\em Ann. Probab.}, 10, 1982.

\bibitem{Rup91}
D.~Ruppert.
\newblock Stochastic approximation.
\newblock In {\em Handbook of sequential analysis}, volume 118 of {\em Statist.
  Textbooks Monogr.}, pages 503--529. Dekker, New York, 1991.

\end{thebibliography}

\end{document}